\documentclass[10pt, leqno]{amsart}
\usepackage{amsmath, amsfonts, amssymb, amsxtra, multicol}

\usepackage[utf8]{inputenc}

\usepackage[english]{babel}
\usepackage{hyperref}
\usepackage[mathscr]{eucal}
\usepackage{tikz-cd}
\usepackage{comment}
\usepackage{enumitem}
\usepackage{seqsplit}

\theoremstyle{plain}
\newtheorem{tw}{Theorem}[section]

\newtheorem {lem} [tw]{Lemma}
\newtheorem {prop}[tw] {Proposition}

\newtheorem{cor}[tw]{Corollary}

\theoremstyle{definition}

\newtheorem {rem} [tw]{Remark}

\newcommand{\cst}{\ifmmode\mathrm{C}^*\else{$\mathrm{C}^*$}\fi}

\newcommand{\M}{\mathsf{M}}
\newcommand{\h}{\mathsf{H}}
\newcommand{\hu}{\mathsf{H}_U}

\newcommand{\hap}{\mathsf{H}_{\mathbb{R}}^{ap}}
\newcommand{\hwm}{\mathsf{H}_{\mathbb{R}}^{wm}}
\newcommand{\hr}{\mathsf{H}_{\mathbb{R}}}

\newcommand{\hc}{\mathsf{H}_\mathbb{C}}
\newcommand{\kr}{\mathsf{K}_{\mathbb{R}}}

\newcommand{\fqh}{\mathcal{F}_q(\mathsf{H}_U)}
\newcommand{\hqn}{\mathsf{H}_U^{\otimes_q^n}}

\newcommand{\cq}[1]{c({#1})}

 \newcommand{\N}{\mathbb{N}}
\newcommand{\R}{\mathbb{R}}
\newcommand{\C}{\mathbb{C}}

\DeclareMathOperator{\dom}{Dom}
\newcommand{\alg} {\text{alg}}

\DeclareMathOperator{\flip}{flip}

\allowdisplaybreaks

\begin{document}

\title{Fullness of $q$-Araki-Woods factors}
\begin{abstract}
    The $q$-Araki-Woods factor associated to a group of orthogonal transformations on a real separable Hilbert space $\mathsf{H}_{\mathbb{R}}$ is full as soon as $\dim\mathsf{H}_{\mathbb{R}}\geq 2$.
\end{abstract}

\author[Manish Kumar]{Manish Kumar}  

\email{mkumar@impan.pl}
\address{Institute of Mathematics of the Polish Academy of Sciences,
	ul.~\'Sniadeckich 8, 00--656
 Warsaw, Poland}

 \author[Simeng Wang]{Simeng Wang}  

\email{simeng.wang@hit.edu.cn}
\address{Institute for Advanced Study in Mathematics, Harbin Institute of Technology,
Harbin, 150001, China}



\maketitle

\maketitle

\section{Introduction}
\noindent
A von Neumann factor $\M$ is called {\em full} if, for any bounded net ${x_i}$ in $\M$ such that $\lim_{i}\|\varphi(\cdot x_i)-\varphi(x_i\cdot)\|_{\M_*}=0$ for all $\varphi\in \M_*$, there exists a bounded net ${\lambda_i}$ of scalars such that $\lim_i(x_i-\lambda_i)=0$ $*$-strongly \cite{Con}. Here, $\M_*$ denotes the predual of $\M$ i.e. the space of normal linear functionals on $\M$. A full factor is always {\em non-injective} or {\em non-hyperfinite} i.e. it cannot be approximated by finite-dimensional subalgebras. For {separable} type II$_1$ factors, the negation of fullness is equivalent to Murray and von Neumann's notion of Property Gamma \cite{Con2}. Connes  established several other equivalent conditions for fullness in type II$_1$ factors in his seminal work \cite{Con2} on the classification program of injective factors. Some of these equivalent conditions, such as the spectral gap property, do not seem to generalize easily to type III factors. However, in a major breakthrough recently, Marrakchi \cite{Mar} proved that fullness for {separable} type III$_1$ factors is equivalent to the assertion that the $C^*$-algebra generated by $\M$ and its commutant $\M'$, represented in the standard form, contains compact operators. The focus of this paper is to utilize one of these equivalent conditions to establish, in full generality, that the $q$-Araki-Woods factors are full.

Recall that the $q$-Araki-Woods factors, denoted as $\Gamma_q(\hr, U_t)$, are associated with a number $q\in (-1,1)$ and a strongly continuous one-parameter group of orthogonal transformations $\{U_t\}_{t\in\R}$ on a real Hilbert space $\hr$ (see the construction in Section \ref{sec:preliminaries}). These von Neumann algebras were introduced by Hiai \cite{Hiai}, combining the quasi-free deformations of Shlyakhtenko \cite{Shl} and $q$-deformations of Bo\.zejko-Speicher \cite{BS}. While various structural properties are known for the free Araki-Woods factors $\Gamma_0(\hr, U_t)$, such as fullness, solidity, strong solidity, the absence of Cartan subalgebras, and complete metric approximation properties \cite{Shl, HRau, HRic, BHV}, many of these properties remain elusive for non-zero $q$, with only a few of the relevant results available \cite{ABW, BM, BMRW, KSW, Nou, SW}. For instance, the  fundamental question of factoriality of $\Gamma_q(\hr, U_t)$ was only recently addressed in complete generality in \cite{KSW}, with some major success achieved earlier in \cite{BMRW}. Furthermore, the factors $\Gamma_q(\hr, U_t)$ are known to be non-injective \cite{KSW, Nou}. Drawing inspiration from the free/undeformed case, it is then natural to ask whether $\Gamma_q(\hr, U_t)$ is full as well for non-zero $q$. This problem has been considered earlier in \cite{Sni, HI, KSW} with   results obtained in various specific cases. In this paper, we address the remaining unresolved cases and hence present the following result in full generality.

\begin{tw}\label{thm: the main result on fullness}
 Let $q\in (-1,1)$,  let $\hr$ be a real separable Hilbert space with $\dim\hr\geq2$,  and let $\{U_t\}_{t\in \R}$ be a strongly continuous  group of orthogonal transformations on $\hr$. Then  the $q$-Araki-Woods factor $\Gamma_q(\hr, U_t)$ is full. 
\end{tw}
We provide a concise overview of previous attempts and highlight our  contribution to the proof. When the representation $\{U_t\}_{t\in\mathbb{R}}$ is trivial and $q\neq 0$, \'Sniady \cite{Sni} proved the fullness of $\Gamma_q(\hr)$ for a large-dimensional real Hilbert space $\hr$, dependent on the value of $q$ (including the case when $\dim\hr=\infty$); additionally, Miyagawa-Speicher \cite{MS} established the fullness of $\Gamma_q(\hr)$ for a finite-dimensional Hilbert space $\hr$ with $\dim\hr\geq 2$. On the other hand, there are mainly  two  papers, \cite{KSW} and \cite{HI}, which have dealt with the fullness of $\Gamma_q(\hr, U_t)$ for $q\neq 0$ and $\{U_t\}_{t\in\mathbb{R}}$ non-trivial. In order to describe them, it is crucial to first mention that the Hilbert space $\hr$ can be decomposed into direct sums of {\em almost periodic} $\hap$ and {\em weakly mixing} $\hwm$ subspaces, where $\hap$ is the direct sum of all finite-dimensional invariant subspaces of $\{U_t\}_{t\in\mathbb{R}}$, and $\hwm$ encompasses the rest. When $\hr$ is finite-dimensional with $\dim\hr\geq2$, $\Gamma_q(\hr, U_t)$ is shown to be full in \cite{KSW}. The proof involves establishing the non-injectivity of $\Gamma_q(\hr, U_t)$ by combining results on conjugate variables and  {general} findings of Nelson \cite{Nel}. It then proceeds to demonstrate Ozawa's AO property for $\Gamma_q(\hr, U_t)$ before invoking the {general} results of Houdayer-Raum \cite{HRau} to conclude fullness. In contrast, Houdayer and Isono \cite{HI} show that if $\hwm\neq0$ or if the point spectrum, the set of  of eigenvalues of $A$, is infinite and bounded, then $\Gamma_q(\hr, U_t)$ is full. Here, $A$ denotes the analytic generator of the one-parameter group of unitaries $\{U_t\}_{t\in\mathbb{R}}$, extended linearly to the complexification of $\hr$. Their proof follows the ideas of \cite{Sni} from the tracial case, i.e. when $\{U_t\}_{t\in\mathbb{R}}$ is trivial. However, their approach encounters difficulties when there is a sequence of eigenvalues growing sufficiently fast to infinity. In this paper, we overcome such challenges by adopting a  different approach.

Our arguments involve a meticulous evaluation of certain Wick products in the centralizer of $\Gamma_q(\hr, U_t)$ associated with vectors in a $2\times2$ block of $\hr$. As in  \cite{BMRW}, we make a crucial use of the non-triviality of the group $\{U_t\}_{t\in\mathbb{R}}$; however, it is worth noting that the methods in \cite{BMRW}  only work for the \textit{weak} convergence of the concerning sequences, which does not apply to the study of fullness. In our approach, we establish several {\em norm} convergence results involving the Wick products and creation/annihilation operators. The key novelty in these convergence results lies in the utilization of an `averaging technique', where one will observe that the sequences themselves may fail to converge. For instance, when $\xi$ and $\eta$ are two unit vectors in $\hr$, the bounded sequence $T_n=(1-q)^nc(\xi)^nc(\eta)^n$ never converges in norm, although it does converge to $0$ in weak operator topology, where $c(\xi)$ denotes the creation operator on the (deformed) Fock space. Nonetheless, we demonstrate that the Ces\'aro sum of $T_n$ converges to $0$ in norm with the additional assumption that $\xi,\eta$ are orthogonal to each other (see Proposition \ref{prop:limit of cesaro sum of C(xi)^kc(eta)^kT_k}). The main idea then is to carefully select such vectors $\xi$ and $\eta$, and employ these convergence results to prove fullness (see Theorem \ref{thm: the main result on fullness}). It is crucial to emphasize that the methods employed in this paper are not applicable to the cases considered in the previous two attempts \cite{KSW, HI}. Thus, a combination of both \cite{KSW} and \cite{HI} is necessary to conclusively establish Theorem \ref{thm: the main result on fullness} in full generality.

This paper is organized as follows. We give a brief overview of the construction of $q$-Araki-Woods factors in Section \ref{sec:preliminaries}. Section \ref{sec: technical analysis} is the core of the article where we carry our main computations on various convergence results. Finally, we put all our efforts together in Section \ref{sec:main result} in proving the main result as mentioned above.

\section{Preliminaries on $q$-Araki-Woods factors}\label{sec:preliminaries}
We briefly describe  construction of the $q$-Araki-Woods  factors and refer the readers to \cite{Hiai} for a detailed treatment. As a convention, all Hilbert spaces are separable and inner products are linear in the second variable. For any real or complex Hilbert space $H$, ${B}(H)$ denotes the algebra of bounded linear operators on $H$.

 Let $\hr$ be a real Hilbert space, and let $\{U_t\}_{t\in\mathbb{R}}$ be a  strongly continuous orthogonal representation of $\R$ on $\hr$. Consider the complexification $ \hc= \hr\otimes_\R\C$ of $\hr$, whose inner product is denoted by $\langle\cdot,\cdot\rangle_{\hc}$.  
 Identify $\hr$ inside $ \hc$ by $\hr\otimes 1$, so that $  \hc=\hr+i\hr$.
Given $\xi=\xi_1+i\xi_2\in  \hc$ for some $\xi_1,\xi_2\in\hr$, write the canonical involution $\bar{\xi}$ as
\begin{equation*}
    \bar{\xi}=\xi_1-i\xi_2.
\end{equation*}
We  extend $\{U_t\}_{t\in\mathbb{R}}$ to a strongly continuous one parameter group of unitaries on the  complexification $ \hc$, which is still denoted by $\{U_t\}_{t\in\mathbb{R}}$. Let $A$ denote the positive, self-adjoint, and non-singular  (unbounded) operator on $\hc$ obtained via Stone's theorem, which satisfies $U_t=A^{it}$ for all $t\in\R$. Note that $\overline{A\xi}=A^{-1}\bar{\xi}$ for any $\xi$ in the domain of $A$. In particular, $\sigma(A)\cap (0,\infty)$ is symmetric around $1$ i.e. for $\lambda\neq 0$, $\lambda\in\sigma(A)$ if and only if $1/\lambda\in\sigma(A)$, where $\sigma(A)$ is the spectrum of $A$.

Following Shlyakhtenko \cite{Shl},  consider the deformed inner product $\langle\cdot,\cdot\rangle_U$ on $\hc$ defined by  
\begin{equation*}
    \langle\xi,\zeta\rangle_U=\langle 2(1+A^{-1})^{-1}\xi, \zeta\rangle\quad\forall\xi,\zeta\in \hc.
\end{equation*}
We denote the completion of $\hc$ with respect to $\langle\cdot,\cdot\rangle_U$  by $\hu$. The space $(\hr,\|\cdot\|_{ \hr})$ is naturally identified as a subspace of $\hu$ such that $\hr+i\hr$ is dense in $\hu$ and $\hr\cap i\hr=\{0\}$.
For $q\in (-1,1)$, consider the {\em $q$-Fock space} $\fqh$ given by 
\begin{equation*}
    \fqh=\C\Omega\oplus\bigoplus_{n\geq 1}\hqn,
\end{equation*}
where $\Omega$ is a distinguished unit vector in $\C$ (called {\em vacuum vector}) and $\hqn$ is the completion of the algebraic tensor product $\h^{\otimes_\alg^n}$ equipped with the following inner product:
\begin{equation*}
    \langle \xi_1\otimes\ldots\otimes\xi_n,\zeta_1\otimes\ldots\otimes\zeta_n\rangle_q=\sum_{\pi\in S_n}q^{|\pi|}\prod_{i=1}^n\langle\xi_i,\zeta_{\pi(i)}\rangle_U
\end{equation*}
for all $\xi_i,\zeta_i\in\hu$, $1\leq i\leq n$. Here $S_n$ denotes the symmetric group on $n$ elements and $|\pi|$ denotes the number of inversions for the permutation $\pi\in S_n$. That the above assignment yields a genuine inner  product was proved in  \cite{BS}. 
On  $\fqh$, we define for each $\xi\in \hu$  operators $c(\xi)$ (called {\em left q-creation}) and  $c^*(\xi)$ (called {\em left q-annihilation}) by the following assignments:
\begin{align*}
   &c(\xi)\Omega=\xi,\\
   &c(\xi)(\xi_1\otimes\ldots\otimes\xi_n)=\xi\otimes\xi_1\otimes\ldots\otimes\xi_n,\;\;\;\mbox{ and }\\ 
 &c^*(\xi)\Omega=0,\\
 &c^*(\xi)(\xi_1\otimes\ldots\otimes\xi_n)=\sum_{i=1}^nq^{i-1}\langle\xi,\xi_i\rangle\xi_1\otimes\ldots\otimes\hat{\xi_i}\otimes\ldots\otimes\xi_n, 
\end{align*}
for all $\xi_1\otimes\ldots\otimes\xi_n\in\hqn$. Here $\hat{\xi_i}$ notates the omitted $i$-th letter from the tensor. That $c(\xi)$ and $c^*(\xi)$ extend to bounded linear operators on $\fqh$  being adjoint to each other is a well known fact (see \cite{BS}).
In a similar fashion, one defines the {\em right $q$-creation} and {\em right $q$-annihilation} operators $r(\xi)$ and $r^*(\xi)$ respectively for all $\xi\in \hu$ as follows:
\begin{align*}
   &r(\xi)\Omega=\xi,\\
   &r(\xi)(\xi_1\otimes\ldots\otimes\xi_n)=\xi_1\otimes\ldots\otimes\xi_n\otimes\xi,\;\;\;\mbox{ and }\\ 
 &r^*(\xi)\Omega=0,\\
 &r^*(\xi)(\xi_1\otimes\ldots\otimes\xi_n)=\sum_{i=1}^nq^{n-i}\langle\xi,\xi_i\rangle\xi_1\otimes\ldots\otimes\hat{\xi_i}\otimes\ldots\otimes\xi_n,
\end{align*}
for all $\xi_1\otimes\ldots\otimes\xi_n\in\hqn$. Again $r(\xi)$ and $r^*(\xi)$ are bounded operators which are adjoint to each other. The creation and annihilation operators satisfy the following $q$-commutation relations:
\begin{equation}\label{eq:q-commutation relations}
\begin{aligned}
&c^*(\xi)c(\eta)-qc(\eta)c^*(\xi)=\langle\xi,\eta\rangle_U\\
&r^*(\xi)r(\eta)-qr(\eta)r^*(\xi)=\langle\xi,\eta\rangle_U
\end{aligned}
\end{equation}
for any $\xi,\eta\in \hu$. We shall repeatedly use the above commutation relations without referring to them. In particular, the most useful use of these relations will be the following consequence:
\begin{align*}
c^*(\xi)^mc(\eta)^n=q^{mn}c(\eta)^nc^*(\xi)^m
\end{align*}
whenever $\xi,\eta\in\hu$ with $\langle\xi,\eta\rangle_U=0$ and $m,n\geq 1$.

We now consider the  von Neumann algebra introduced by Hiai \cite{Hiai} (following Shlyakhtenko \cite{Shl} and Bo\.zejko-Speicher \cite{BS}), denoted $\Gamma_q(\hr, U_t)$, to be the von Neumann subalgebra of ${B}(\mathcal{F}_q(\hu))$ generated by the set $\{W(\xi); \xi\in\hr\}$, where 
\begin{equation*}\label{eq:definition of sq(xi)}
    W(\xi)= \cq\xi+c^*(\xi),\;\;\;\;\;\;\;\xi\in\hr.
\end{equation*}
The algebra $\Gamma_q(\hr, U_t)$ is  called the {\em $q$-deformed Araki-Woods von Neumann algebra}. If the representation $\{U_t\}_{t\in\mathbb{R}}$ is clear from the context, we shall denote the algebra $\Gamma_q(\hr, U_t)$ simply by $\M_q$. We shall always assume that $\dim\hr\geq 2$. Since it is now known that $\M_q$ is a {\em factor} i.e. $\M_q$ has trivial center (see \cite{KSW, Hiai}), we will call it  a {\em  $q$-Araki-Woods factor.}

The vacuum vector $\Omega$ is  cyclic and separating  for the factor $\M_q$, and hence induces a faithful normal state: $\varphi(\cdot)=\langle \Omega, \cdot\;\Omega\rangle_q$ on $\M_q$, called the {\em $q$-quasi free state}. It is immediate to verify  that the GNS Hilbert space associated to $\varphi$ is $\fqh$.
Since $\Omega$ is a separating vector for both $\M_q$ and $\M_q'$, there exist for each $\xi\in \M_q\Omega$ and $\zeta\in \M_q'\Omega$  unique operators $x_\xi\in \M_q$ and $x'_\zeta\in \M_q'$ such that $\xi=x_\xi\Omega$ and $\zeta=x_\zeta'\Omega$. Write
\begin{equation*}
    W(\xi)=x_\xi, \;\mbox{ and }\;W_r(\zeta)=x'_\zeta.
\end{equation*}
Note that the notation $W(\xi)$ is in line with  the one introduced above for $\xi\in\hr$. The operators $W(\xi)$ and $W_r(\zeta)$ are respectively called {\em left} and {\em right Wick operators}. It is known that $\xi_1\otimes\ldots\otimes\xi_n\in \M_q\Omega$ for all $\xi_i\in \h_\C, 1\leq i\leq n$ and $\zeta_1\otimes\ldots\otimes\zeta_m\in \M_q'\Omega$ for all $\zeta_j\in  \h_\C\cap \dom(A^{-1/2}), 1\leq j\leq m$. By definition, the operators ${W}(\xi)$ and $W_r(\zeta)$ commute for all $\xi\in \M_q\Omega$ and $\zeta\in \M_q'\Omega$.  In fact, the commutant of $\M_q$ in  $B(\fqh)$ is given by 
\begin{equation*}
\Gamma_q(\hr, U_t)'= \{W_r(\eta);\eta\in\hr'\}^{''}
\end{equation*}
where $\hr'=\{\eta\in\h_\C; \langle \xi,\eta\rangle_U\in\R \mbox{ for all }\xi\in\hr\}$. We have $\hc=\hr'+i\hr'$ and $\hr\cap i\hr=\{0\}$.  It is easy to verify that
\begin{equation*}
    W(\xi)=\cq\xi+c^*({\bar{\xi}})\;\;\mbox{ and }\;\;W_r(\zeta)=r(\zeta)+r^*(\bar{\zeta}^r)
\end{equation*}
for any $\xi\in  \hc$ and $\zeta\in \hc\cap \dom(A^{-1/2})$, where $\bar{\zeta}^r $ denotes the real part of $\zeta$ with respect to the decomposition $\hc=\hr'+i\hr'$.
The formula for the left and right Wick operators are given as follows:

\begin{prop}\label{prop:formula for Wick operators}
 For any vectors $\xi_1,\ldots,\xi_n\in \hr+i\hr$,     the Wick-formula for operators in $\M_q$ is given by
    \begin{align*}        W(\xi_1\otimes\cdots\otimes\xi_n)&=\sum_{i=0}^n\sum_{\sigma\in S_{n, i}}q^{|\sigma|}c(\xi_{\sigma(1)})\cdots c(\xi_{\sigma(i)})c^*(\bar{\xi}_{\sigma(i+1)})\cdots c^*(\bar{\xi}_{\sigma(n)})
    \end{align*}
    where $S_{n,j}$ is the set of those permutations in $S_n$ which are increasing on ${\{1,\ldots,j\}}$ and on $\{j+1,\ldots,n\}$.
\end{prop}

\begin{prop}\label{prop:formula for right Wick product}
 For any vectors $\xi_1,\ldots,\xi_n\in \hr'+i\hr'$,     the Wick-formula for operators in $\M_q'$ is given by
    \begin{align*}        W_r(\xi_n\otimes\cdots\otimes\xi_1)&=\sum_{i=0}^n\sum_{\sigma\in S_{n, i}}q^{|\flip\circ\sigma|}r(\xi_{\sigma(1)})\cdots r(\xi_{\sigma(n-i)})r^*(\bar{\xi}_{\sigma(n-i+1)}^r)\cdots r^*(\bar{\xi}_{\sigma(n)}^r)\\
   & =\sum_{i=0}^n\sum_{\sigma\in S_{n, n-i}}q^{|\sigma|}r(\xi_{n+1-\sigma(n)})\cdots r(\xi_{n+1-\sigma(i+1)})r^*(\bar{\xi}_{n+1-\sigma(i)}^r)\cdots r^*(\bar{\xi}_{n+1-\sigma(1)}^r)
    \end{align*}
    where $S_{n,j}$ is the set of those permutations in $S_n$ which are increasing on ${\{1,\ldots,j\}}$ and on $\{j+1,\ldots,n\}$.
\end{prop}
We adopt the following notations for our convenience.
\begin{enumerate}
    \item $\xi_1\ldots\xi_n=\xi_1\otimes\ldots\otimes\xi_n$, for $\xi_i\in\hu$, $1\leq i\leq n$.
    \item $\xi^n=\xi^{\otimes n}$, $n\geq 1$ and $\xi^0=\Omega$ for all $\xi\in\hu$. 
    \item $d_0=1$, $d_n=\prod_{j=1}^n(1-q^j)$ for $n\geq 1$ and $d_\infty=\prod_{j=1}^\infty(1-q^j)$
    \item $C_q=\prod_{j=1}^\infty\frac{1}{1-|q|^j}$
    \item $[n]_q=\sum_{j=0}^{n-1}q^j=\frac{1-q^n}{1-q}$
    \item $[n]_q!=\prod_{j=1}^n[j]_q=\frac{d_n}{(1-q)^n}$
    \item $\begin{pmatrix}
        n\\k
    \end{pmatrix}_q=\frac{[n]_q!}{[k]_q![n-k]_q!}$ for any $1\leq k\leq n$.
\end{enumerate}

The following are well known inequalities (see \cite{BS}), which we shall use very frequently throughout the paper:
\begin{enumerate}
\item $d_n\leq C_q$,
\item $\frac{1}{d_n}\leq C_q$.
\item $\|c(\xi)^n\|\leq \sqrt{C_q}\|\xi^n\|_q$ for all $\xi\in\hu$.
\end{enumerate}

The following statements are either well-known or easy to verify.

\begin{lem}\label{lem:formula for c*}
Given $f\in \hu$, and $\xi\in\hu^{\otimes_q^m}, \eta\in \hu^{\otimes_q^n}$ for some $m,n\geq 1$, we have
\begin{align*}
    c^*(f)(\xi\otimes\eta)=(c^*(f)\xi)\otimes\eta+q^m\xi\otimes c^*(f)\eta
\end{align*}
\end{lem}

\begin{lem}\label{lem:upper bound for c(xi^n)}
For any $\xi\in\hu$ and $n\geq1$, we have $\|c(\xi)^n\|\leq  C_q(1-q)^{-\frac{n}{2}}\|\xi\|^\frac{n}{2}.$
\end{lem}
\begin{proof} This follows by recalling that $\|c(\xi)^n\|\leq \sqrt{C_q}\|\xi^n\|_q$, as well as by noting that $[n]_q!= d_n(1-q)^{-n}\leq C_q(1-q)^{-n}$ for all $n\geq 1$. 
\end{proof}

\begin{cor}
\label{lem:bounded of sequence c(e)n}
For any $\xi\in \hu$, the sequences 
 \begin{enumerate}
     \item $(1-q)^{\frac{n}{2}}\|\xi\|_U^{-n}c(\xi)^n$
     \item  $(1-q)^{\frac{n}{2}}\|\xi\|_U^{-n}r(\xi)^n$
 \end{enumerate}
are bounded, both with an upper bound given by $C_q$.
\end{cor}
We end this section by recalling the general structure of a one parameter group of orthogonal operators. Let $\{U_t\}_{t\in\mathbb{R}}$ be a strongly continuous orthogonal representation of $\R$ on a real Hilbert space $\hr$ and let $A$ denote the analytic generator of the complexification of $\{U_t\}_{t\in\mathbb{R}}$. Let $\h_\R^{ap}$ denote the closed real subspace generated by the eigenvectors of $A$ and let $\h_\R^{wm}={\h_\R^{ap}}^{\perp}$. The subspaces $\h_\R^{ap}$ and $\h_\R^{wm}$ are reducing under $A$ and every $\{U_t\}_{t\in\mathbb{R}}$, and they are respectively called the {\em almost periodic} and {\em weakly mixing} part of $\hr$. Similarly, ${U_t}_{|_{\hap}}$ and ${U_t}_{|_{\hwm}}$ are called the almost periodic and weakly mixing part of $\{U_t\}_{t\in\mathbb{R}}$ respectively.

Note that whenever $\hwm\neq0$, it must be infinite-dimensional. On the other hand,  the almost periodic part decomposes as follows (see \cite{Shl}):
\begin{equation*}
    (\hap, {U_t}_{|_{\hap}})=\left(\bigoplus_{j=1}^{N_1}(\R, id)\right)\oplus\left(\bigoplus_{j=1}^{N_2}\hr^k, U_t^k\right)
\end{equation*}
for some $0\leq N_1, N_2\leq \aleph_0$, where for all $1\leq k\leq N_2$ we have
\begin{equation*}
    \hr^k=\R^2, \;\;\;\; U_t^k=\begin{pmatrix}
    \cos(t\log \lambda_k)& -\sin(t\log\lambda_k)\\
    \sin(t\log\lambda_k)&\cos(t\log\lambda_k)
    \end{pmatrix} 
\end{equation*}
for some $\lambda_k\in (0,1)$.

\section{Technical analysis}\label{sec: technical analysis}

This section contains the rigorous technical analysis, where we will prove a series of Lemmas and Propositions about various convergence results. The most crucial results are Propositions \ref{prop:convergence of S_n=left wick formula}, \ref{prop:limit of cesaro sum of C(xi)^kc(eta)^kT_k} and \ref{prop: convergence of r^*(xi)r^*(eta)c(xi)c(eta)}. We recommend that readers first review the statements of these Propositions and then proceed directly to the proof of the main theorem in the next section to understand the reasons behind these calculations. The notations used in Proposition \ref{prop:convergence of S_n=left wick formula} are mentioned just before Lemma \ref{lem:boundedness of sequence S_n and R_n}.

\begin{lem}\label{lem:expression for c*(xi)mc(xi)n}
    For any fixed non-zero vector $\xi\in\hu$, $n\in\N$ and $1\leq m\leq n$, we have
    \begin{align}\label{eq:c^*(xi)^mc(xi)^m}
        c^*(\xi)^mc(\xi)^n=\frac{[n]_q!}{[n-m]_q!}\|\xi\|_U^{2m}c(\xi)^{n-m}+\sum_{k=0}^{m-1}\frac{[n]_q!}{[n-k]_q!}q^{n-k}\|\xi\|_U^{2k}c^*(\xi)^{m-k-1}c(\xi)^{n-k}c^*(\xi).
    \end{align}
\end{lem}
\begin{proof}
The proof is by induction on $m$.    For $m=1$,  use Lemma \ref{lem:formula for c*} to write
\begin{align}\label{eq:for m=1}
 c^*(\xi)c(\xi)^n=[n]_q\|\xi\|_U^2c(\xi)^{n-1}+q^nc(\xi)^nc^*(\xi)
 \end{align} 
 which is exactly as in \eqref{eq:c^*(xi)^mc(xi)^m}. Now assume the desired equation to be true for some  $1\leq m\leq n-1$; then calculate
    \begin{align*}
        c^*(\xi)^{m+1}c(\xi)^n&=c^*(\xi)\left(c^*(\xi)^mc(\xi)^n\right)\\
        &=  \frac{[n]_q!}{[n-m]_q!}\|\xi\|_U^{2m}c^*(\xi)c(\xi)^{n-m}+\sum_{k=0}^{m-1}\frac{[n]_q!}{[n-k]_q!}\|\xi\|_U^{2k}q^{n-k}c^*(\xi)^{m-k}c(\xi)^{n-k}c^*(\xi)\\
        &\hspace{3in} (\text{by induction hypothesis})\\
        &=\frac{[n]_q!}{[n-m]_q!}\|\xi\|_U^{2m}\left([n-m]_q\|\xi\|_U^2c(\xi)^{n-m-1}+q^{n-m}c(\xi)^{n-m}c^*(\xi)\right)\\
        &\quad\quad\quad\quad+\sum_{k=0}^{m-1}\frac{[n]_q!}{[n-k]_q!}\|\xi\|_U^{2k}q^{n-k}c^*(\xi)^{m-k}c(\xi)^{n-k}c^*(\xi) \quad\quad\quad(\mbox{by using eq. \eqref{eq:for m=1} for } n-m)\\
        &=\frac{[n]_q!}{[n-m-1]_q!}\|\xi\|_U^{2m+2}c(\xi)^{n-(m+1)}+\sum_{k=0}^{m }\frac{[n]_q!}{[n-k]_q!}\|\xi\|_U^{2k}q^{n-k}c^*(\xi)^{m-k}c(\xi)^{n-k}c^*(\xi).
    \end{align*}
    The proof then ends by the induction argument.
\end{proof}

\begin{cor}\label{cor:expression for c*(bar e)^mc(bar e)^nc(e)^n}
For vectors $\xi,\eta\in \hu$ with $\langle\xi,\eta\rangle_U=0$,  $n\in\N$, $0\leq m\leq n$ and $\ell\geq 1$, we have
\begin{align*}
   c^*(\xi)^mc(\xi)^nc(\eta)^\ell= \frac{[n]_q!}{[n-m]_q!}\|\xi\|_U^{2m}c(\xi)^{n-m}c(\eta)^\ell+q^\ell X
\end{align*}
 where  $X$ is an operator bounded by  $A_q'(1-q)^{-\frac{(m+n+\ell)}{2}}\|\xi\|_U^{m+n}\|\eta\|_U^{\ell}$ for some constant $A_q'$ depending only on $q$.
\end{cor}
\begin{proof}
If $m=0$ then set $X=0$, and if $m\geq1$ then set \[X=\sum_{k=0}^{m-1}\frac{[n]_q!}{[n-k]_q!}q^{n-k}\|\xi\|_U^{2k}c^*(\xi)^{m-k-1}c(\xi)^{n-k}c(\eta)^\ell c^*(\xi).\] Obtain the required expression by multiplying from the right in eq. \eqref{eq:c^*(xi)^mc(xi)^m} of Lemma \ref{lem:expression for c*(xi)mc(xi)n}  by $c(\eta)^\ell$ and then by noting via $q$-commutation relation that $c^*(\xi)c(\eta)^\ell=q^\ell c(\eta)^\ell c^*(\xi)$. Further, we estimate the norm of $X$ using Lemma \ref{lem:upper bound for c(xi^n)} as follows: 
    \begin{align*}
        \|X\|&\leq C_q^4\sum_{k=0}^{m-1}\frac{[n]_q!}{[n-k]_q!}|q|^{n-k}\|\xi\|_U^{2k}\|\xi\|_U^{(m-k-1)+(n-k)+1}\|\eta\|_U^\ell(1-q)^{-\frac{1}{2}(m-k-1)-\frac{1}{2}(n-k)-\frac{\ell}{2}-\frac{1}{2}}\\
        &= C_q^4  \sum_{k=0}^{m-1}\frac{d_n(1-q)^{-n}}{d_{n-k}(1-q)^{-n+k}} |q|^{n-k} \|\xi\|_U^{m+n}\|\eta\|_U^\ell  (1-q)^{-\frac{(m+n+\ell)}{2}+k}\\
        &\leq C_q^6  \|\xi\|_U^{m+n}\|\eta\|_U^\ell  (1-q)^{-\frac{(m+n+\ell)}{2}}\sum_{k=0}^{m-1}|q|^{n-k}\quad\quad\quad\quad(\text{as }d_n, \frac{1}{d_{n-k}}\leq C_q)\\
        &= \|\xi\|_U^{m+n}\|\eta\|_U^\ell  (1-q)^{-\frac{(m+n+\ell)}{2}}\frac{C_q^6}{1-|q|}
    \end{align*}
which gives the required estimate by setting $A_q'=\frac{C_q^6}{1-|q|}$.
\end{proof}

\begin{cor}\label{cor:when m>n}
For vectors $\xi,\eta\in \hu$ with $\langle\xi,\eta\rangle_U=0$, and $m> n\geq1$, $\ell\geq 1$, we have
\begin{align*}
   c^*(\xi)^mc(\xi)^nc(\eta)^\ell= q^\ell Y
\end{align*}
 where  $Y$ is an operator bounded by  $A_q(1-q)^{-\frac{(m+n+\ell)}{2}}\|\xi\|_U^{m+n}\|\eta\|_U^{\ell}$ for some constant $A_q$ depending only on $q$.
\end{cor}    
\begin{proof}
    Using  Corollary \ref{cor:expression for c*(bar e)^mc(bar e)^nc(e)^n},  write
    \begin{align*}
        c^*(\xi)^nc(\xi)^nc(\eta)^\ell=[n]_q!\|\xi\|_U^{2n}c(\eta)^\ell+q^\ell X
    \end{align*}
    where $X$ is an operator bounded by $\frac{C_q^6}{1-|q|}(1-q)^{-n-\frac{\ell}{2}}\|\xi\|_U^{2n}\|\eta
    \|_U^\ell$. By multiplying both sides from the left by $c^*(\xi)^{m-n}$ and using the $q$-commutation relation $c^*(\xi)^{m-n}c(\eta)^\ell=q^{(m-n)\ell}c(\eta)^\ell c^*(\xi)^{(m-n)}$ since $\langle\xi,\eta\rangle_U=0$, we get
    \begin{align*}
        c^*(\xi)^mc(\xi)^nc(\eta)^\ell=q^{(m-n)\ell}[n]_q!\|\xi\|_U^{2n}c(\eta)^\ell c^*(\xi)^{m-n}+q^\ell c^*(\xi)^{m-n}X=q^\ell Y
    \end{align*}
    where $Y=q^{(m-n-1)\ell}[n]_q!\|\xi\|_U^{2n}c(\eta)^\ell c^*(\xi)^{m-n}+ c^*(\xi)^{m-n}X$. It is an easy  verify as before using Lemma \ref{lem:upper bound for c(xi^n)} that \begin{align*}
        \|Y\|&\leq |q|^{(m-n-1)\ell}d_nC_q^2\|\xi\|_U^{m+n}\|\eta\|_U^\ell (1-q)^{-\frac{(m+n+\ell)}{2}}+\frac{C_q^6}{1-|q|}\|\xi\|_U^{m+n}\|\eta\|_U^\ell C_q (1-q)^{-\frac{(m+n+\ell)}{2}}\\
        &\leq A_q \|\xi\|_U^{m+n}\|\eta\|_U^\ell(1-q)^{-\frac{(m+n+\ell)}{2}}
    \end{align*}
    where $A_q=|q|^{(m-n-1)\ell}C_q^3+\frac{C_q^7}{1-|q|}$.
\end{proof}

\begin{prop}\label{prop:limit of cesaro sum of C(xi)^kc(eta)^kT_k}
For any non-zero vectors $\xi,\eta\in \hu$ with $\langle\xi,\eta\rangle_U=0$ and a uniformly bounded sequence $\{T_k\}_{k\geq1}$ of operators, the sequences
\begin{enumerate}
    \item $\frac{1}{n}\sum_{k=1}^n(1-q)^{k}\|\xi\|_U^{-k}\|\eta\|_U^{-k}c(\xi)^kc(\eta)^kT_k$
    \item $\frac{1}{n}\sum_{k=1}^n(1-q)^{k}\|\xi\|_U^{-k}\|\eta\|_U^{-k}r(\xi)^kr(\eta)^kT_k$
    \end{enumerate}
  converge to $0$ in norm as $n\to\infty$. 
\end{prop}
\begin{proof}
We shall prove the convergence of the first sequence, which we call $X_n$. The proof for the second sequence follows similarly. By replacing $\xi$ and $\eta$ with $\frac{\xi}{\|\xi\|_U}$ and $\frac{\eta}{\|\eta\|_U}$ respectively, we may and will  assume without loss of generality that $\xi$ and $\eta$ are unit vectors. Also assume that $\sup_{k\geq1}\|T_k\|= 1$. Now calculate
   \begin{align*}
       X_n^*X_n&=\frac{1}{n^2}\sum_{\ell=1}^n\sum_{k=1}^n(1-q)^{\ell+k}T_\ell^* c^*(\eta)^\ell c^*(\xi)^\ell c(\xi)^kc(\eta)^kT_k\\
       &=X^n_{\ell=k}+X^n_{k<\ell}+X^n_{k>\ell}
   \end{align*}
   where we have divided the sum over three cases: $\ell=k, k<\ell$ and $k>\ell$ i.e.
   \begin{align*}
       &X_{\ell=k}^n= \frac{1}{n^2}\sum_{k=1}^n(1-q)^{2k}T_k^* c^*(\eta)^k c^*(\xi)^k c(\xi)^kc(\eta)^kT_k\\
       & X_{k<\ell}^n=\frac{1}{n^2}\sum_{1\leq k<\ell\leq n}(1-q)^{\ell+k}T_\ell^* c^*(\eta)^\ell c^*(\xi)^\ell c(\xi)^kc(\eta)^kT_k\\
       & X_{\ell<k}^n=\frac{1}{n^2}\sum_{n\geq k>\ell\geq1}(1-q)^{\ell+k}T_\ell^* c^*(\eta)^\ell c^*(\xi)^\ell c(\xi)^kc(\eta)^kT_k.
   \end{align*}
  We now estimate the norm of the three sums separately. Firstly, over $k=\ell$, we estimate the norm using Lemma \ref{lem:upper bound for c(xi^n)} as follows:
   \begin{equation*}
      \|X_{\ell=k}^n\|\leq \frac{1}{n^2}\sum_{k=1	}^n C_q^4=\frac{C_q^4}{n}.
   \end{equation*}
   We next  estimate the norm of $X_{k<\ell}^n$: for $k<\ell$,  we use Corollary \ref{cor:when m>n} to write  $c^*(\xi)^\ell c(\xi)^kc(\eta)^k=q^kX_{k,\ell}$ where $X_{k,\ell}$ is an  operator bounded by $A_q(1-q)^{-\frac{\ell}{2}-k}$. Therefore, we have
   \begin{align*}
    \|X_{k<\ell}^n\|&=   \frac{1}{n^2}\left\|\sum_{1\leq k<\ell\leq n}(1-q)^{\ell+k}c^*(\eta)^\ell T_\ell^* \left(q^kX_{k,\ell}\right)T_k\right\|\\
       &\leq \frac{C_qA_q}{n^2}\sum_{1\leq k<\ell\leq n}|q|^k(1-q)^{\ell+k}(1-q)^{-\frac{\ell}{2}-k}(1-q)^{-\frac{\ell}{2}}\\
       &\leq \frac{C_qA_q}{n^2}\sum_{k,\ell=1}^n|q|^k\\
       &\leq \frac{C_qA_q}{n}\sum_{k=1}^\infty |q|^k\\
       &\leq\frac{1}{n}\cdot \frac{C_qA_q}{1-|q|}.
   \end{align*}
   Similarly,  by symmetry  $X_{k>\ell}^n=(X_{k<\ell}^n)^*$,  $\|X^n_{k>\ell}\|$ has the same estimate  $\|X_{k<\ell}^n\|$. Thus we obtain that 
   \begin{align*}
       \|X_n^*X_n\|\leq \frac{1}{n} \left({C_q^4}+\frac{2C_qA_q}{1-|q|}\right)
   \end{align*}
   which proves that $X_n\to0$ in norm as $n\to\infty$.
\end{proof}

\begin{rem}
    It is very important that we put the operator $T_k$ on the right side in Proposition \ref{prop:limit of cesaro sum of C(xi)^kc(eta)^kT_k} and not on the left. In fact, the operator $\tilde{S}_n=\frac{1}{n}\sum_{k=0}^n(1-q)^k\tilde{T}_kc(\xi)^kc(\eta)^k$ need not converge to $0$ for fixed unit vectors $\xi,\eta\in\hu$ with $\langle\xi,\eta\rangle_U=0$ and bounded operators  $\tilde{T}_k$. See for example Proposition \ref{prop:convergence of S_n=left wick formula} below. 
\end{rem}


We now restrict our attention to a two-dimensional invariant subspace of a given representation $\{U_t\}_{t\in\mathbb{R}}$. Let $\kr=\R^2$ be an invariant subspace of $\{U_t\}_{t\in\mathbb{R}}$ such that ${U_t}_{|_{\kr}}$ is ergodic i.e. has no fixed nonzero vector. Then, as discussed in Section \ref{sec:preliminaries}, there is a scalar $\lambda\in (0,1)$ such that 
\begin{align*}
  {U_t}_{|_{\kr}}=  \begin{pmatrix}
        \cos(t\log\lambda)&-\sin(t\log\lambda)\\
        \sin(t\log\lambda)&\cos(t\log\lambda)
    \end{pmatrix}.
\end{align*}
Set \begin{align*}e=\frac{({\lambda+1})^{\frac{1}{2}}}{2\lambda^{\frac{1}{4}}}\begin{pmatrix}
    1\\i 
\end{pmatrix}\in\mathsf{K}_\C=\C^2.
\end{align*}
If $A$ denotes the analytic generator, then it is immediate to check that
 \[Ae=\lambda^{-1}e,\;\;\;A\bar{e}=\lambda\bar{e},\;\;\;\|e\|_U=\lambda^{-\frac{1}{4}},\;\;\; \|\bar{e}\|_U=\lambda^{\frac{1}{4}} \;\;\;\mbox{ and  }\langle e,\bar{e}\rangle_U=0.\] Moreover, the (right) conjugate of $e$ and $\bar{e}$ with respect to the decomposition in $\hr'+i\hr'$ is given as follows: 
 \[\bar{e}^r=\lambda^{-1}\bar{e}  \;\;\;\mbox{ and }\;\;\;\overline{(\bar{e})}^r=\lambda e.\]
  Indeed it is straightforward to check that $(1+A^{-1})^{-1}(e+\lambda^{-1}\Bar{e})=\frac{1}{1+\lambda}(e+\Bar{e})\in \hr$ and hence $\langle e+\lambda^{-1}\Bar{e},\xi\rangle_U\in \R$ for all $\xi\in \hr$, and similarly check that $\langle e-\lambda^{-1}\Bar{e},\xi\rangle_U\in i\hr$ for all $\xi\in \hr$; this will prove that $e+\lambda^{-1}\Bar{e}\in \hr'$ and $e-\lambda^{-1}\bar{e}\in i\hr'$ and hence $ \bar{e}^r=\lambda^{-1}\bar{e}$.
  
We fix such $\lambda, e$ and $\bar{e}$ for the rest of this section.   The notations used in what follows have been explained in Section \ref{sec:preliminaries}.  Our next main result is Proposition \ref{prop:convergence of S_n=left wick formula}, but first we need two crucial lemmas.   The following lemma is borrowed from \cite{BMRW}. 

\begin{lem}\label{lem:boundedness of sequence S_n and R_n}
    The sequences $\{(1-q)^nW(\bar{e}^ne^n)\}_{n\geq 1}$ and $\{(1-q)^nW_r(\bar{e}^ne^n)\}_{n\geq1}$ are bounded.
\end{lem}
\begin{proof}
   We briefly sketch the proof from \cite[Lemma 3.8]{BMRW} in order to set some notations and record some useful observations. Note that both the sequences have the same norm; indeed we have $W_r(\bar{e}^ne^n)=J_\varphi W(\bar{e}^ne^n)J_\varphi$. We consider the first sequence which by Wick formula is given by
   \begin{align}\label{eq:expression for W(bar e^ne^n)}
W(\Bar{e}^ne^n)=\sum_{k,\ell=0}^nq_{k,\ell}^{(n)}c(\Bar{e})^kc(e)^\ell c^*(e)^{n-k}c^*(\Bar{e})^{n-\ell}
   \end{align}
  where $q_{k,\ell}^{(n)}=\sum_{J_1, J_2}q^{c(J, J^c)}$ with the sum running over all possible subsets $J_1\subseteq\{1,\ldots,n\}$ of cardinality $k$ and $J_2\subseteq\{n+1,\ldots,2n\}$ of cardinality $\ell$ such that $J=J_1\cup J_2$. Here $c(J, J^c)$ denotes the number of crossings of the partition $J\cup J^c$ of $\{1,2,\ldots,2n\}$ (see \cite[Lemma 3.8]{BMRW} for more explanation of the terms $q_{k,\ell}^{(n)}$). Then we have $|q_{k,\ell}^{(n)}|\leq C_q^2|q|^{(n-k)\ell}$ and \begin{align*}
     (1-q)^n \sum_{k,\ell=0}^n |q_{k,\ell}^{(n)}|\;\|c(\Bar{e})^k\|\;\|c(e)^\ell\|\;\| c^*(e)^{n-k}\|\;\|c^*(\Bar{e})^{n-\ell}\|\leq C_q^6\sum_{k,\ell=0}^n |q|^{(n-k)\ell}\lambda^{\frac{k-\ell}{2}}.
   \end{align*}
   The important observation is that $\sum_{k,\ell=0}^n |q|^{(n-k)\ell}\lambda^{\frac{k-\ell}{2}}\leq \frac{1}{(1-|q|)(1-\lambda^{\frac{1}{2}})}+\frac{(n+1)|q|^n\lambda^{-\frac{1}{2}}}{1-|q|^n\lambda^{-\frac{1}{2}}}$ as $n\to\infty$, which means that there is a constant 
   $B(q,\lambda)$ depending on $q$ and $\lambda$ such that
    \[\sum_{k,\ell=0}^n |q|^{(n-k)\ell}\lambda^{\frac{k-\ell}{2}}\leq B(q,\lambda)\] for all $n\geq 1$.  This observation will be useful for us later as well (also see Lemma \ref{lem:convergence of the sclar sequence s_n} below).
\end{proof}

\begin{lem}\label{lem:convergence of the sclar sequence s_n}
    For any $q\in (-1,1)$ and $\lambda\in (0,1)$, the sequence \[s_n=\frac{1}{n}\sum_{m=1}^n\sum_{\ell=1}^m\sum_{k=0}^{m-1}q^{(m-k)\ell}\lambda^{\frac{k-\ell}{2}}\]
    converges to $0$ as $n\to \infty$. 
\end{lem}
\begin{proof}
Assume that $q\geq0$. Write $s_n$ as follows:
\begin{align}
    s_n&\notag=\frac{1}{n}\sum_{m=1}^n\sum_{0\leq k<\ell\leq m}q^{(m-k)\ell}\lambda^{\frac{k-\ell}{2}}+\frac{1}{n}\sum_{m=1}^n\sum_{1\leq \ell\leq k\leq m-1}q^{(m-k)\ell}\lambda^{\frac{k-\ell}{2}}\\
    &\label{eq:s_n divided into two parts}=\frac{1}{n}\sum_{m=1}^n\sum_{j=-m}^{-1}\sum_{k=0}^{m+j}q^{(m-k)(k-j)}\lambda^{\frac{j}{2}}+\frac{1}{n}\sum_{m=1}^n\sum_{j=0}^{m-1}\sum_{k=j+1}^{m-1}q^{(m-k)(k-j)}\lambda^{\frac{j}{2}}
\end{align}
where we get the last equality by putting $j=k-\ell$. Now we estimate the first quantity in \eqref{eq:s_n divided into two parts}, where we follow the arguments as in the proof of \cite[Lemma 3.9]{BMRW}: 
\begin{align*}
    \frac{1}{n}\sum_{m=1}^n\sum_{j=-m}^{-1}\sum_{k=0}^{m+j}q^{(m-k)(k-j)}\lambda^{\frac{j}{2}}&=\frac{1}{n}\sum_{m=1}^n\sum_{j=1}^{m}\lambda^{-\frac{j}{2}}\sum_{k=0}^{m-j}q^{(m-k)(k+j)}\leq \frac{1}{n}\sum_{m=1}^n\sum_{j=1}^{m}\lambda^{-\frac{j}{2}}(m-j+1)q^{mj}.
\end{align*}
Choose $n_0$ large enough such that $q^{n_0}\lambda^{-\frac{1}{2}}\leq \frac{1}{2}$; hence for $m\geq n_0$ we have $q^m\lambda^{-\frac{1}{2}}\leq\frac{1}{2}q^{m-n_0}$ so that the above sum further yields
\begin{align*}
  \frac{1}{n}\sum_{m=1}^n\sum_{j=-m}^{-1}\sum_{k=0}^{m+j}q^{(m-k)(k-j)}\lambda^{\frac{j}{2}}&\leq \frac{1}{n}\sum_{m=1}^{n_0}\sum_{j=1}^{m}\lambda^{-\frac{j}{2}}(m-j+1)q^{mj} +\frac{1}{n}\sum_{m=n_0+1}^{n}\sum_{j=1}^{m}(m+1)\left(q^m\lambda^{-\frac{1}{2}}\right)^{j}\\
  &\leq\frac{ C(q,\lambda,n_0)}{n}+
   \frac{1}{n}\sum_{m=n_0+1}^n(m+1)\sum_{j=1}^m\left(\frac{q^{m-n_0}}{2}\right)^j
   \\
  & \leq\frac{ C(q,\lambda,n_0)}{n}+
   \frac{1}{n}\sum_{m=1}^n(m+1)\frac{q^{m-n_0}}{2-q^{m-n_0}}\\
  &\to0 \mbox{ as }n\to \infty,
\end{align*}
where $C(q,\lambda,n_0)$ is a constant independent from $n$, and where we have also used the fact that the sequence $\left\{(m+1)\frac{q^m}{2q^{n_0}-q^m}\right\}_{m\geq 1}$ converges to $0$ as $m\to\infty$ and hence its Ces\'aro sum also converges to $0$.
On the other hand,  the second quantity in \eqref{eq:s_n divided into two parts} is estimated as follows:
    \begin{align*}
       \frac{1}{n}\sum_{m=1}^n\sum_{j=0}^{m-1}\sum_{k=j+1}^{m-1}q^{(m-k)(k-j)}\lambda^{\frac{j}{2}} 
       &= \frac{1}{n}\sum_{j=0}^{n-1}\lambda^{\frac{j}{2}} \sum_{k=j+1}^{n-1}\sum_{m=k+1}^nq^{(m-k)(k-j)}\\
       &=\frac{1}{n}\sum_{j=0}^{n-1}\lambda^{\frac{j}{2}} \sum_{k=j+1}^{n-1}q^{-k(k-j)}\sum_{m=k+1}^nq^{m(k-j)}\\
       &= \frac{1}{n}\sum_{j=0}^{n-1}\lambda^{\frac{j}{2}} \sum_{k=j+1}^{n-1}q^{-k(k-j)}\left(\frac{q^{(k+1)(k-j)}-q^{(n+1)(k-j)}}{1-q^{k-j}}\right)\\
       & \leq \frac{1}{n}\sum_{j=0}^{n-1}\lambda^{\frac{j}{2}} \sum_{k=j+1}^{n-1}\frac{q^{k-j}}{1-q^{k-j}}\\
       &\leq \frac{1}{1-\lambda^{\frac{1}{2}}}\left(\frac{1}{n}\sum_{k=1}^{n}\frac{q^k}{1-q^k}\right)\\
       &\to 0 \mbox{ as }n\to\infty
    \end{align*}
     as  the sequence $\{\frac{q^k}{1-q^k}\}_{k\geq 1}$  converges to $0$ as $k\to\infty$, and so does its Ces\'aro sum. Thus we have proved that $s_n\to0$ as $n\to\infty$.
\end{proof}

\begin{prop}\label{prop:convergence of S_n=left wick formula}
Let $e,\bar{e}$ and $\lambda$ be as above. The sequence \[S_n=\frac{1}{n}\sum_{m=1}^n(1-q)^{2m}W(\Bar{e}^me^m)c(\Bar{e})^mc(e)^m\] is convergent in norm to the operator $S$ given by
\begin{equation}\label{eq:expression for S}
    S=\sum_{k=0}^\infty\frac{d_\infty}{d_k}(1-q)^{k}\;\lambda^{\frac{k}{2}}c(\Bar{e})^kTc(e)^k
\end{equation}
where  $T
$ is the norm limit of the operator $(1-q)^n\lambda^{\frac{n}{2}}c^*(e)^nc(e)^n$ which is  positive and invertible. Moreover, $S$ is invertible if $\lambda<\left(1+C_q^4\right)^{-2}$.
\end{prop}
\begin{proof}
 By \cite[Lemma 3.4]{BMRW}, the operator $T$ exists  and  is both positive and invertible.
First check that the series defined in \eqref{eq:expression for S} above is norm convergent so that $S$ is a well-defined bounded operator. Indeed if we set $Z_k=\frac{d_\infty}{d_k}(1-q)^{k}\;\lambda^{\frac{k}{2}}c(\Bar{e})^kTc(e)^k$ for $k\geq 0$ then by using Lemma \ref{lem:bounded of sequence c(e)n} we have
\begin{align*}
\sum_{k=0}^\infty\|Z_k\|&\leq   \sum_{k=0}^\infty\;\frac{d_\infty}{d_k}(1-q)^{k}\;\lambda^{\frac{k}{2}}\|c(\Bar{e})^k\|\;\|T\|\;\|c(e)^k\|\\
   &\leq   d_\infty C_q\;\|T\|\sum_{k=0}^\infty(1-q)^{k}\;\lambda^{\frac{k}{2}} \left(C_q(1-q)^{-\frac{k}{2}}\lambda^{\frac{k}{4}}\right)\left(C_q(1-q)^{-\frac{k}{2}}\lambda^{-\frac{k}{4}}\right)\\
   &=d_\infty C_q^3\; \|T\|\sum_{k=0}^\infty \lambda^{\frac{k}{2}}\\
   &<\infty.
   \end{align*} 
To prove the convergence  of $S_n$, write using \eqref{eq:expression for W(bar e^ne^n)}
\begin{align*}
S_n=\frac{1}{n}\sum_{m=1}^n(1-q)^{2m}\sum_{k,\ell=0}^mq_{k,\ell}^{(m)}c(\Bar{e})^kc(e)^\ell c^*(e)^{m-k}c^*(\Bar{e})^{m-\ell}c(\Bar{e})^mc(e)^m
\end{align*}
where $q_{k,\ell}^{(m)}$ are the scalars as in the proof of Lemma \ref{lem:boundedness of sequence S_n and R_n} which satisfy $|q^{(m)}_{k,\ell}| \leq C_q^2|q|^{(m-k)\ell}$.
Let us understand the terms $c^*(\Bar{e})^{m-\ell}c(\Bar{e})^mc(e)^m$ in the sum above. Use Corollary \ref{cor:expression for c*(bar e)^mc(bar e)^nc(e)^n} to  write
\begin{align*}
   c^*(\Bar{e})^{m-\ell}c(\Bar{e})^mc(e)^m= \frac{[m]_q!}{[\ell]_q!}\lambda^{\frac{m-\ell}{2}}c(\bar{e})^{\ell}c(e)^m+q^mX_{m,\ell}
\end{align*}
where $X_{m,\ell}$ is an operator bounded by $(1-q)^{\frac{(-3m+\ell)}{2}}\lambda^{\frac{m-\ell}{4}}A_q$ for the constant $A_q:=\frac{C_q^6}{1-|q|}$ depending only on $q$. In particular, we get
\begin{align*}
    S_n&=
    \frac{1}{n}\sum_{m=1}^n(1-q)^{2m}\sum_{k,\ell=0}^mq_{k,\ell}^{(m)}\left(\frac{[m]_q!}{[\ell]_q!}\lambda^{\frac{m-\ell}{2}}\right)c(\Bar{e})^kc(e)^\ell c^*(e)^{m-k}c(\Bar{e})^{\ell}c(e)^m\\
    &\quad\quad\quad\quad+\frac{1}{n}\sum_{m=1}^nq^m(1-q)^{2m}\sum_{k,\ell=0}^mq_{k,\ell}^{(m)}c(\Bar{e})^kc(e)^\ell c^*(e)^{m-k}X_{m,\ell}.
\end{align*}
We now calculate the upper bound for the last sum above:
\begin{align*}
    \frac{1}{n}\sum_{m=1}^n|q|^m(1-q)^{2m}&\left\|\sum_{k,\ell=0}^mq_{k,\ell}^{(m)}c(\Bar{e})^kc(e)^\ell c^*(e)^{m-k}X_{m,\ell}\right\| \\
    &\leq \frac{1}{n}\sum_{m=1}^n|q|^m(1-q)^{2m}C_q^5A_q\sum_{k,\ell=0}^m|q|^{(m-k)\ell}\lambda^{\frac{k-\ell-(m-k)}{4}}(1-q)^{-\frac{(k+\ell+m-k)}{2}}(1-q)^{\frac{(-3m+\ell)}{2}}\lambda^{\frac{m-\ell}{2}}\\
    &=\frac{1}{n}C_q^5A_q\sum_{m=1}^n|q|^m\sum_{k,\ell=0}^m|q|^{(m-k)\ell}\lambda^{\frac{k-\ell}{2}}\\
    &\leq \frac{1}{n}\frac{C_q^5A_qB(q,\lambda)}{1-|q|}
\end{align*}
where $B(q,\lambda)$ is an upper bound for the sum $\sum_{k,\ell=0}^m|q|^{(m-k)\ell}\lambda^{\frac{k-\ell}{2}}$ such that $B(q,\lambda)$ depends only on $\lambda$ and $q$, and is independent of $m$ (see the proof of Lemma \ref{lem:boundedness of sequence S_n and R_n}). Thus we have
\begin{align*}
S_n&=
    \frac{1}{n}\sum_{m=1}^n(1-q)^{2m}\sum_{k,\ell=0}^mq_{k,\ell}^{(m)}\left(\frac{[m]_q!}{[\ell]_q!}\lambda^{\frac{m-\ell}{2}}\right)c(\Bar{e})^kc(e)^\ell c^*(e)^{m-k}c(\Bar{e})^{\ell}c(e)^m+O\left(\frac{1}{n}\right)\\
    &=X_n+Y_n+Z_n+O\left(\frac{1}{n}\right)
    \end{align*}
    where the above sum is divided over $\ell=0, k=m$ and the rest i.e. we write 
    \begin{align*}
    &X_n=\frac{1}{n}\sum_{m=1}^n(1-q)^{2m}\sum_{k=0}^mq_{k,0}^{(m)}[m]_q!\lambda^{\frac{m}{2}}c(\Bar{e})^k c^*(e)^{m-k}c(e)^m\\
    &Y_n=\frac{1}{n}\sum_{m=1}^n(1-q)^{2m}\sum_{\ell=0}^mq_{m,\ell}^{(m)}\frac{[m]_q!}{[\ell]_q!}\lambda^{\frac{m-\ell}{2}}c(\Bar{e})^mc(e)^\ell c(\Bar{e})^\ell c(e)^m\\
    &Z_n= \frac{1}{n}\sum_{m=1}^n(1-q)^{2m}\sum_{\ell=1}^m\sum_{k=0}^{m-1}q_{k,\ell}^{(m)}\left(\frac{[m]_q!}{[\ell]_q!}\lambda^{\frac{m-\ell}{2}}\right)c(\Bar{e})^kc(e)^\ell c^*(e)^{m-k}c(\Bar{e})^{\ell}c(e)^m.
\end{align*}
But the norm of $Z_n$ can be estimated using Lemma \ref{lem:upper bound for c(xi^n)} as follows:
\begin{align*}
    \|Z_n\|&\leq \frac{1}{n}\sum_{m=1}^n(1-q)^{2m}\sum_{\ell=1}^m\sum_{k=0}^{m-1}C_q^2|q|^{(m-k)\ell}\left(\frac{d_m}{d_\ell}\lambda^{\frac{m-\ell}{2}}\right)C_q^5(1-q)^{-2m}\lambda^{\frac{k-m}{2}}\\
    &\leq \frac{C_q^9}{n}\sum_{m=1}^n\sum_{\ell=1}^n\sum_{k=0}^{m-1}|q|^{(m-k)\ell}\lambda^{\frac{k-\ell}{2}}\\
    &= C_q^9s_n
   \end{align*}
 where $s_n$ is the sequence as in Lemma \ref{lem:convergence of the sclar sequence s_n}; hence $Z_n\to0$ as $n\to\infty$ by the same Lemma. We now show that $Y_n$
converges to 0 as $n\to\infty$. We follow the same trick as in the proof of Proposition \ref{prop:limit of cesaro sum of C(xi)^kc(eta)^kT_k}. Calculate
\begin{align*}
    Y_n^*Y_n&=\frac{1}{n^2}\sum_{m,m'=1}^n(1-q)^{2m+2m'}\\
    &\cdot\sum_{\ell=0}^m\sum_{\ell'=0}^{m'}q_{m,\ell}^{(m)}q_{m',\ell'}^{(m')}\frac{[m]_q![m']_q}{[\ell]_q![\ell']_q!}\lambda^{\frac{m+m'-\ell-\ell'}{2}}c^*(e)^mc^*(\Bar{e})^\ell c^*(e)^\ell c^*(\bar{e})^{m}c(\Bar{e})^{m'}c(e)^{\ell'} c(\bar{e})^{\ell'} c(e)^{m'}\\
    &=W^n_{m=m'}+W^n_{m>m'}+W^n_{m<m'}
\end{align*}
where we divide the sum over three cases: $m=m', m<m'$ and $m>m'$.  Over $m=m'$, we again use some similar estimates to calculate the following:
\begin{align*}\|W_{m=m'}^n\|&= \frac{1}{n^2}\left\|\sum_{m=1}^n(1-q)^{4m}\sum_{\ell,\ell'=0}^mq_{m,\ell}^{(m)}q_{m,\ell'}^{(m)}\frac{([m]_q!)^2}{[\ell]_q![\ell']_q!}\lambda^{\frac{2m-\ell-\ell'}{2}}c^*(e)^mc^*(\Bar{e})^\ell c^*(e)^\ell c^*(\bar{e})^{m}c(\Bar{e})^{m}c(e)^{\ell'} c(\bar{e})^{\ell'} c(e)^{m}\right\|\\
&\leq \frac{C_q^{12}}{n^2}\sum_{m=1}^n(1-q)^{4m}\sum_{\ell,\ell'=0}^m\frac{d_m^2(1-q)^{-2m}}{d_\ell d_{\ell'}(1-q)^{-\ell-\ell'}}\lambda^{\frac{2m-\ell-\ell'}{2}}(1-q)^{-2m-\ell-\ell'}\quad\quad\quad\quad(\text{since }q_{m,\ell}^{(m)}\leq C_q^2)\\
&\leq \frac{C_q^{16}}{n^2}\sum_{m=1}^n\sum_{\ell,\ell'=0}^m\lambda^{\frac{2m-\ell-\ell'}{2}}\\
&\leq\frac{1}{n} \frac{C_q^{16}}{(1-\sqrt{\lambda})^2}.
\end{align*}
Next consider the sum over  $m>m'$ (which has the same norm as the part $m<m'$ by considering the adjoint and the symmetric nature of the sum). For $m>m'$, use Corollary \ref{cor:when m>n} to write
\[c^*(\bar{e})^{m}c(\Bar{e})^{m'}c(e)^{\ell'}=q^{\ell'}X_{m,m',\ell'}\] where $X_{m,m',\ell'}$ is an operator bounded by $A_q(1-q)^{-\frac{(m+m'+\ell')}{2}}\lambda^{\frac{m+m'-\ell'}{4}}$ for some constant $A_q$. Therefore, the sum over $m>m'$ is estimated as
\begin{align*}
 &\|W^n_{m>m'}\|\\
 &=\frac{1}{n^2}\left\|\sum_{n\geq m>m'\geq 1}^n(1-q)^{2m+2m'}\sum_{\ell=0}^m\sum_{\ell'=0}^{m'}q_{m,\ell}^{(m)}q_{\ell',m'}^{(m')}\frac{[m]_q![m']_q}{[\ell]_q![\ell']_q!}\lambda^{\frac{m+m'-\ell-\ell'}{2}}c^*(e)^mc^*(\Bar{e})^\ell c^*(e)^\ell \left(q^{\ell'}X_{m,m',\ell'}\right) c(\bar{e})^{\ell'} c(e)^{m'}\right\|\\ 
    &\leq \frac{C_q^9A_q}{n^2}\sum_{n\geq m>m'\geq1}(1-q)^{2m+2m'}\sum_{\ell=0}^m\sum_{\ell'=0}^{m'}|q|^{\ell'}\frac{d_md_{m'}(1-q)^{-m-m'}}{d_\ell d_{\ell'}(1-q)^{-\ell-\ell'}}\lambda^{\frac{m+m'-\ell-\ell'}{2}}(1-q)^{-m-m'-\ell-\ell'}\quad(\text{as }q_{m,\ell}^{(m)}\leq C_q^2)\\\
    & \leq \frac{C_q^{13}A_q}{n^2}\sum_{m,m'=1}^n\sum_{\ell=0}^m\sum_{\ell'=0}^{m'}|q|^{\ell'}\lambda^{\frac{m+m'-\ell-\ell'}{2}}\\
    &\leq \frac{C_q^{13}A_q}{n^2}\frac{n}{1-\sqrt{\lambda}}\sum_{m'=1}^n\sum_{\ell'=0}^{m'}|q|^{\ell'}\lambda^{\frac{m'-\ell'}{2}}\\
&\leq \frac{C_q^{13}A_q}{n(1-\sqrt{\lambda})}\sum_{m'=1}^n\left(\sum_{\ell'=0}^{\frac{m}{2}}|q|^{\ell'}\lambda^{\frac{m'-\ell'}{2}}+\sum_{\ell'=\frac{m'}{2}}^m|q|^{\ell'}\lambda^{\frac{m'-\ell'}{2}}\right)\\
&\leq \frac{C_q^{13}A_q}{n(1-\sqrt{\lambda})}\left(\sum_{m'=1}^n\lambda^{\frac{m'}{2}}\sum_{\ell'=0}^{\frac{m}{2}}|q|^{\ell'}+\sum_{m'=1}^n|q|^{\frac{m'}{2}}\sum_{\ell'=\frac{m'}{2}}^m\lambda^{\frac{m'-\ell'}{2}}\right)\\
   &\leq \frac{C_q^{13}A_q}{n(1-\sqrt{\lambda})}\left(\frac{1}{(1-\sqrt{\lambda})(1-|q|)}+\frac{1}{(1-\sqrt{|q|})(1-\sqrt{\lambda})}\right)\\
   &\leq \frac{1}{n}\frac{2C_q^{13}A_q}{(1-\sqrt{\lambda})^2(1-\sqrt{|q|})}.
\end{align*}
A similar upper bound can be given for $\|W_{m<m'}^n\|$. Thus we have shown that 
\begin{align*}
    \|Y_n^*Y_n\|\leq \frac{1}{n}\left(\frac{C_q^{16}}{(1-\sqrt{\lambda})^2}+\frac{4C_q^{13}A_q}{(1-\sqrt{\lambda})^2(1-\sqrt{|q|})}\right)
\end{align*}
and hence $\lim_{n\to\infty}Y_n=0$ in norm.
 Finally, we conclude that 
\begin{align*}
    \lim_{n\to\infty}S_n&=\lim_{n\to\infty}X_n=\lim_{n\to\infty}\frac{1}{n}\sum_{m=1}^n(1-q)^{2m}\sum_{k=0}^mq_{k,0}^{(m)}[m]_q!\lambda^{\frac{m}{2}}c(\Bar{e})^k c^*(e)^{m-k}c(e)^m\\
    &=\lim_{n\to\infty}\frac{1}{n}\sum_{m=1}^n\sum_{k=0}^mq_{k,0}^{(m)}d_m(1-q)^k\lambda^{\frac{k}{2}}c(\Bar{e})^k T_{m-k}c(e)^k
\end{align*}
where $T_p=(1-q)^p\lambda^{\frac{p}{2}}c^*(e)^p c(e)^p$ for any $p\geq 1$. Let $T=\lim_{p\to\infty}T_p$ in norm (which exists by \cite[Lemma 3.4]{BMRW} as mentioned above). Also for $0\leq k\leq m$, it is straightforward to verify that $q_{k,0}^{(m)}=\begin{pmatrix}m\\k
        \end{pmatrix}_q$ (see for example, \cite[Corollary 2.8]{BKS}); so that for each $k\geq0$, we have
    \[\lim_{m\to\infty}q_{k,0}^{(m)}=\lim_{m\to\infty}\begin{pmatrix}m\\k
        \end{pmatrix}_q=\lim_{m\to\infty}{\frac{d_m}{d_kd_{m-k}}}=\frac{1}{d_k}.\]
        Finally, since $\lim_{m\to\infty}d_m=d_\infty$ and $\lim_{m\to\infty}T_{m-k}=T$ for all fixed $k\geq 1$, we use dominated convergence theorem to conclude that
        \begin{align*}
            \lim_{m\to\infty}\sum_{k=0}^mq_{k,0}^{(m)}d_m(1-q)^k\lambda^{\frac{k}{2}}c(\Bar{e})^k T_{m-k}c(e)^k=\sum_{k=0}^\infty \frac{d_\infty}{d_k}(1-q)^k\lambda^{\frac{k}{2}}c(\Bar{e})^k Tc(e)^k.
        \end{align*}
        In particular, the Ces\'aro sum of the above sequence has the same limit; hence we get
        \begin{align*}
            \lim_{n\to\infty}S_n=\lim_{n\to\infty}\frac{1}{n}\sum_{m=1}^n\sum_{k=0}^mq_{k,0}^{(m)}d_m(1-q)^k\lambda^{\frac{k}{2}}c(\Bar{e})^k T_{m-k}c(e)^k=\sum_{k=0}^\infty \frac{d_\infty}{d_k}(1-q)^k\lambda^{\frac{k}{2}}c(\Bar{e})^k Tc(e)^k=S.
        \end{align*}
       Finally, to prove the invertibility of $S$, we note that
       \begin{align*}
           d_{\infty}^{-1}T^{-1}S-I=\sum_{k=1}^\infty \frac{1}{d_k}(1-q)^k\lambda^{\frac{k}{2}}T^{-1}c(\Bar{e})^k Tc(e)^k
       \end{align*}
    so that
       \begin{align*}
           \|d_\infty^{-1}T^{-1}S-I\|\leq C_q^2\sum_{k=1}^\infty\frac{1}{d_k}\lambda^{\frac{k}{2}}\|T^{-1}\|\;\|T\|&\leq C_q^3\|T^{-1}\|\;\|T\|\frac{\lambda^{\frac{1}{2}}}{{1-\lambda^{\frac{1}{2}}}}\leq C_q^4\frac{\lambda^{\frac{1}{2}}}{{1-\lambda^{\frac{1}{2}}}}
       \end{align*}
       where we have used the fact that $\|T^{-1}\|\;\|T\|\leq \frac{1}{d_\infty}\leq C_q$ from \cite[Lemma 3.4]{BMRW}.
       In particular, $d_\infty^{-1}T^{-1}S$ is invertible if $C_q^4\frac{\lambda^{\frac{1}{2}}}{{1-\lambda^{\frac{1}{2}}}}<1$, which is satisfied by the given value of $\lambda$ in the hypothesis.
\end{proof}

\begin{rem}\label{rem:estimate for the norm of S{-1}} 
 1. Since $\|d_\infty^{-1}T^{-1}S-1\|\leq C_q^4\frac{\sqrt{\lambda}}{1-\sqrt{\lambda}}$, we have  that
 \begin{align*}
     \|d_\infty TS^{-1}\|=\|(I-(I-d_\infty^{-1} T^{-1}S))^{-1}\|\leq \frac{1}{1-\|I-d_\infty^{-1}T^{-1}S\|}\leq \frac{1}{1-C_q^4\frac{\sqrt{\lambda}}{1-\sqrt{\lambda}}}=\frac{1-\sqrt{\lambda}}{1-(1+C_q^4)\sqrt{\lambda}}
 \end{align*}
 and since $T^{-1}\leq 2d_\infty^{-1}\leq 2C_q$ (\cite[Lemma 3.4]{BMRW}), we must have
 \begin{align*}
     \|S^{-1}\|\leq \|d_\infty^{-1}T^{-1}\|\;\|d_\infty TS^{-1}\|\leq \frac{2C_q^2(1-\sqrt{\lambda})}{1-(1+C_q^4)\sqrt{\lambda}}.
 \end{align*}
 2. Although we have shown that the Ces\'aro sum of the sequence $\{W(\bar{e}^me^m)c(\Bar{e})^mc(e)^m\}_{m\geq 1}$  converges, the sequence itself does not converge. For example, this can be checked for $q=0$ where the main obstruction is that the sequence $\{\sum_{\ell=0}^mc(\Bar{e})^mc(e)^\ell c^*(\Bar{e})^{m-\ell}c(\bar{e})^mc(e)^m\}_{m\geq 1}$ is not convergent. This is one very important reason for us to consider Ces\'aro sum of such bounded sequence.
\end{rem}

We now prove some convergence results involving both the left and right creation/annihilation operators. For this purpose, we introduce the  operator $Q$ defined by
\begin{align*}
Q=\sum_{n=0}^\infty q^nP_n
\end{align*}
where $P_n$ denotes the projection in $B(\fqh)$ onto the $n$th particle space $\hu^{\otimes_q^n}$, and the sum above converges in norm.  Note that for any $f\in \hu^{\otimes_q^n}$, we have $Qf=q^nf$. The following is an easy verification
\begin{align}\label{eq: commutation relation between r*(xi) and c(eta)}
r^*(\xi)c(\eta)=\langle \xi,\eta\rangle_UQ+c(\eta)r^*(\xi)
\end{align}
for any $\xi,\eta\in\hu$. In particular, $r^*(\xi)$ and $c(\eta)$ commute if $\xi$ and $\eta$ are orthogonal. Moreover, we have
\begin{align}\label{eq:commutation between Q and c(xi) and r(xi)}
Qc(\xi)=qc(\xi)Q,\quad Qr(\xi)=qr(\xi)Q
\end{align}

\begin{lem}\label{lem:expression for r*(xi)^mc(eta)^n}
For any $\xi,\eta\in \hu$ and $m,n\in\mathbb{N}$, we have the following:
\begin{align*}
r^*(\xi)^mc(\eta)^n=\sum_{k=0}^{\min\{m,n\}}\begin{pmatrix}n\\k
        \end{pmatrix}_q\frac{[m]_q!}{[m-k]_q!}\langle\xi,\eta\rangle_U^kc(\eta)^{n-k}Q^kr^*(\xi)^{m-k}.
\end{align*}
\end{lem}
\begin{proof}
We may and we will assume that $\langle\xi,\eta\rangle_U\neq 0$, otherwise there is nothing to prove since  $r^*(\xi)c(\eta)=c(\eta)r^*(\xi)$ whenever $\langle\xi,\eta\rangle_U=0$. By replacing $\xi$ with $\frac{\xi}{\langle\eta,\xi\rangle_U}$ (or $\eta$ with $\frac{\eta}{\langle\xi,\eta\rangle_U})$, we will further assume that $\langle\xi,\eta\rangle_U=1$.

 First, we prove the expression for $r^*(\xi)c(\eta)^n$ by induction on $n$. Indeed for $n=1$, the expression follows from \eqref{eq: commutation relation between r*(xi) and c(eta)}. Assuming this to be true for $n-1$, we again use  \eqref{eq: commutation relation between r*(xi) and c(eta)} and \eqref{eq:commutation between Q and c(xi) and r(xi)}  to  obtain
\begin{equation}\label{eq:formula for r*(xi)c(eta)^n}
\begin{aligned}
r^*(\xi)c(\eta)^n&=\left(r^*(\xi)c(\eta)\right)c(\eta)^{n-1}=Qc(\eta)^{n-1}+c(\eta)(r^*(\xi)c(\eta)^{n-1})\\
&=q^{n-1}c(\eta)^{n-1}Q+c(\eta)\left(c(\eta)^{n-1}r^*(\xi)+[n-1]_qc(\eta)^{n-2}Q\right)\\
&=c(\eta)^nr^*(\xi)+[n]_qc(\eta)^{n-1}Q
\end{aligned}
\end{equation}
as required. Now we prove the result for general $m$ and $n$ such that $n\geq m$. The case $n<m$ follows similarly. So now we fix $n$.   The proof proceeds by induction on $m$.  For $m=1$, the  expression is proved above. By assuming the expression to be true for $m-1$, we calculate
\begin{align*}
r^*(\xi)^mc(\eta)^n&=r^*(\xi)(r^*(\xi)^{m-1}c(\eta)^n)\\
&=r^*(\xi)\sum_{k=0}^{m-1}\begin{pmatrix}n\\k
        \end{pmatrix}_q\frac{[m-1]_q!}{[m-1-k]_q!}c(\eta)^{n-k}Q^kr^*(\xi)^{m-1-k}\\
        &=\sum_{k=0}^{m-1}\begin{pmatrix}n\\k
        \end{pmatrix}_q\frac{[m-1]_q!}{[m-1-k]_q!}\left(c(\eta)^{n-k}r^*(\xi)+[n-k]_qc(\eta)^{n-k-1}Q\right)Q^kr^*(\xi)^{m-1-k}\\
        &\hspace{3in}(\text{using the expression for } r^*(\xi)c(\eta)^{n-k}\text{ from }\eqref{eq:formula for r*(xi)c(eta)^n})\\
        &=\sum_{k=0}^{m-1}\begin{pmatrix}n\\k
        \end{pmatrix}_q\frac{[m-1]_q!}{[m-1-k]_q!}q^kc(\eta)^{n-k}Q^kr^*(\xi)^{m-k}\\
        &\quad\quad+\sum_{k=0}^{m-1}\begin{pmatrix}n\\k
        \end{pmatrix}_q\frac{[m-1]_q!}{[m-1-k]_q!}[n-k]_qc(\eta)^{n-k-1}Q^{k+1}r^*(\xi)^{m-1-k}\\
       &= \sum_{k=0}^{m-1}\begin{pmatrix}n\\k
        \end{pmatrix}_q\frac{[m-1]_q!}{[m-1-k]_q!}q^kc(\eta)^{n-k}Q^{k}r^*(\xi)^{m-k}\\
        &\quad\quad+\sum_{k=1}^{m}\begin{pmatrix}n\\k-1
        \end{pmatrix}_q\frac{[m-1]_q!}{[m-k]_q!}[n-k+1]_qc(\eta)^{n-k}Q^{k}r^*(\xi)^{m-k}\\&\hspace{3in} (\text{using }r^*(\xi)Q^k=q^kQ^kr^*(\xi) \text{ from }\eqref{eq:commutation between Q and c(xi) and r(xi)})\\
        &=c(\eta)^nr^*(\xi)^m+\sum_{k=1}^{m-1}\left(\begin{pmatrix}n\\k
        \end{pmatrix}_q\frac{[m-1]_q!}{[m-1-k]_q!}q^k+\begin{pmatrix}n\\k-1
        \end{pmatrix}_q\frac{[m-1]_q!}{[m-k]_q!}[n-k+1]_q\right)c(\eta)^{n-k}Q^{k}r^*(\xi)^{m-k}\\
       & \quad\quad\quad+\begin{pmatrix}n\\m-1
        \end{pmatrix}_q[m-1]_q![n-m+1]_qc(\eta)^{n-m}Q^m\\
        &=\sum_{k=0}^{m}\begin{pmatrix}n\\k
        \end{pmatrix}_q\frac{[m]_q!}{[m-k]_q!}c(\eta)^{n-k}Q^kr^*(\xi)^{m-k}
\end{align*}
where to get the last expression, we have used the following calculation:
\begin{align*}
\begin{pmatrix}n\\k
        \end{pmatrix}_q\frac{[m-1]_q!}{[m-1-k]_q!}q^k&+\begin{pmatrix}n\\k-1
        \end{pmatrix}_q\frac{[m-1]_q!}{[m-k]_q!}[n-k+1]_q\\
        &=\frac{[n]_q!}{[n-k]_q![k-1]_q!}\frac{[m-1]_q!}{[m-1-k]_q!}\left(\frac{q^k}{[k]_q}+\frac{1}{[m-k]_q}\right)\\
        &=\frac{[n]_q!}{[n-k]_q![k-1]_q!}\frac{[m-1]_q!}{[m-1-k]_q!}\left(\frac{q^k[m-k]_q+[k]_q}{[k]_q[m-k]_q}\right)\\
        &=\frac{[n]_q!}{[n-k]_q![k-1]_q!}\frac{[m-1]_q!}{[m-1-k]_q!}\left(\frac{[m]_q}{[k]_q[m-k]_q}\right)\\
        &=\begin{pmatrix}n\\k
        \end{pmatrix}_q\frac{[m]_q!}{[m-k]_q!}.
\end{align*}
This completes the proof.
\end{proof}

The following proposition  can be considered a stronger version of \cite[Lemma 3.11]{BMRW} in the sense of norm convergence  replacing the weak convergence.

\begin{prop}\label{prop: convergence of r^*(xi)r^*(eta)c(xi)c(eta)}
For any non-zero vectors $\xi,\eta\in \hu$ with $\langle\xi,\eta\rangle_U=0$, 
    we have $\lim_{n\to\infty}R_n=d_\infty^2 P_{\Omega}$ in norm, where $R_n$ is the  operator \[\frac{1}{n}\sum_{m=1}^n(1-q)^{2m} \|\xi\|_U^{-2m}\|\eta\|_U^{-2m} r^*(\xi)^mr^*(\eta)^mc(\xi)^mc(\eta)^m\]
   and $P_\Omega$ denotes the projection onto the vacuum space $\C\Omega$.
\end{prop}
\begin{proof}
Assume without loss of generality that $\xi,\eta$ are unit vectors. First of all, note that $r^*(\eta)$ and $c(\xi)$ commute as $\langle\xi,\eta\rangle_U=0$, so we can also write
\begin{align*}
  R_n=  \frac{1}{n}\sum_{m=1}^n(1-q)^{2m}  r^*(\xi)^mc(\xi)^mr^*(\eta)^mc(\eta)^m.
\end{align*}
Use Lemma \ref{lem:expression for r*(xi)^mc(eta)^n} to expand $R_n$ as follows:
\begin{align*}
R_n&=\frac{1}{n}\sum_{m=1}^n(1-q)^{2m}\left(\sum_{k=0}^{m}\begin{pmatrix}m\\k
        \end{pmatrix}_q\frac{[m]_q!}{[m-k]_q!}c(\xi)^{m-k}Q^kr^*(\xi)^{m-k}\right)\left(\sum_{\ell=0}^{m}\begin{pmatrix}m\\\ell
        \end{pmatrix}_q\frac{[m]_q!}{[m-\ell]_q!}c(\eta)^{m-\ell}Q^\ell r^*(\eta)^{m-\ell}\right)
        \\&=\frac{1}{n}\sum_{m=1}^n(1-q)^{2m}\sum_{k,\ell=0}^{m}\begin{pmatrix}m\\k
        \end{pmatrix}_q\begin{pmatrix}m\\\ell
        \end{pmatrix}_q\frac{[m]_q!^2}{[m-k]_q![m-\ell]_q!}c(\xi)^{m-k}Q^kc(\eta)^{m-\ell}r^*(\xi)^{m-k}Q^\ell r^*(\eta)^{m-\ell}\\
        &\hspace{4in} (\text{since }r^*(\xi)c(\eta)=c(\eta)r^*(\xi))\\
        &=\frac{1}{n}\sum_{m=1}^n(1-q)^{2m}\sum_{k,\ell=0}^{m}q^{k(m-\ell)+\ell(m-k)}\begin{pmatrix}m\\k
        \end{pmatrix}_q\begin{pmatrix}m\\\ell
        \end{pmatrix}_q\frac{[m]_q!^2}{[m-k]_q![m-\ell]_q!}c(\xi)^{m-k}c(\eta)^{m-\ell}Q^{k+\ell}r^*(\xi)^{m-k} r^*(\eta)^{m-\ell}\\
        &\hspace{1in} (\text{since }Q^kc(\eta)^{m-\ell}=q^{k(m-\ell)}c(\eta)^{m-\ell}Q^k \text{ and }r^*(\xi)^{m-k}Q^\ell=q^{(m-k)\ell}Q^\ell r^*(\xi)^{m-k}\text{ from }\eqref{eq:commutation between Q and c(xi) and r(xi)})\\
        &=\frac{1}{n}\sum_{m=1}^nd_m^2Q^{2m}+\frac{1}{n}\sum_{m=1}^n(1-q)^{2m}c(\xi)^mc(\eta)^mr^*(\xi)^mr^*(\eta)^m+X_n
\end{align*}
where the sum is divided over $(k,\ell=m)$, $(k,\ell=0)$ and the rest, i.e. $X_n$ is the operator
\begin{align*}
X_n=\frac{1}{n}\sum_{m=1}^n(1-q)^{2m}\sum_{\underset{0<k+\ell<2m}{k,\ell=0}}^{m}q^{m(k+\ell)-2k\ell}\begin{pmatrix}m\\k
        \end{pmatrix}_q\begin{pmatrix}m\\\ell
        \end{pmatrix}_q\frac{[m]_q!^2}{[m-k]_q![m-\ell]_q!}c(\xi)^{m-k}c(\eta)^{m-\ell}Q^{k+\ell}r^*(\xi)^{m-k} r^*(\eta)^{m-\ell}.
\end{align*}
Thus we have divided $R_n$ as a sum of three terms, and we will calculate norm limit of each term. To get the limit of the first term,  first verify   that $Q^{2m}$ converges to $P_\Omega$ as $m\to\infty$, and therefore the average $\frac{1}{n}\sum_{m=1}^nd_m^2Q^{2m}$ converges to $d_\infty^2P_{\Omega}$ in norm. The second term
\[\frac{1}{n}\sum_{m=1}^n(1-q)^{2m}c(\xi)^mc(\eta)^mr^*(\xi)^mr^*(\eta)^m\]
 converges to $0$ in norm by Proposition \ref{prop:limit of cesaro sum of C(xi)^kc(eta)^kT_k}. We finally show that $X_n$ converges to $0$ in norm, which will finish the proof.
 Indeed first note that $m(k+\ell)-2k\ell\geq m$ whenever $0<k+\ell<2m$ (one can observe this fact by  assuming without loss generality that $\ell\leq k$, and then considering various cases  such as $0\leq \ell\leq \frac{m}{2} \;\&\; 1\leq k\leq \frac{m}{2}$ or $\frac{m}{2}\leq \ell<m\;\&\; \frac{m}{2}<k\leq m$ etc); therefore we have $q^{m(k+\ell)-2k\ell}\leq |q|^m$. This observation along with Lemma \ref{lem:upper bound for c(xi^n)} and the fact that $\|Q\|\leq 1$ yield the norm estimate of $X_n$ as follows:
\begin{align*}
\|X_n\|&\leq \frac{C_q^2}{n}\sum_{m=1}^n|q|^{m}(1-q)^{2m}\sum_{\underset{0<k+\ell<2m}{k,\ell=0}}^{m}\frac{[m]_q!^2}{[k]_q![\ell]_q![m-k]_q![m-\ell]_q!}\frac{[m]_q!^2}{[m-k]_q![m-\ell]_q!}\left([m-k]_q![m-\ell]_q!\right)\\
&\leq \frac{C_q^2}{n}\sum_{m=1}^n|q|^m\sum_{{k,\ell=0}}^{m}\frac{d_m^4}{d_kd_\ell d_{m-k}d_{m-\ell}}\\
&\leq \frac{C_q^{10}}{n}\sum_{m=1}^n\sum_{{k,\ell=0}}^{m}|q|^m\\
&\leq \frac{C_q^{10}}{n}\sum_{m=1}^nm^2|q|^m.
\end{align*}
The sequence $\{m^2q^m\}_{m\geq 1}$ converges to $0$ as $m\to\infty$, and so does its average; hence we conclude that $X_n$ converges to $0$ in norm as $n\to\infty$. 
\end{proof}

We end this section by giving an expression for the right Wick product $W_r(\bar{e}^ne^n)$, which can easily be proved by utilizing Proposition \ref{prop:formula for right Wick product}.

\begin{cor}\label{cor:formula for W_rbar{e}^ne^n}
For $e,\bar{e}$ and $\lambda$ as before, the right Wick product $W_r(\Bar{e}^ne^n)$ is given as
    \begin{align*}
       W_r(\bar{e}^ne^n)&=\sum_{k,\ell=0}^nq^{(n)}_{k,\ell}r(e)^kr(\bar{e})^\ell r^*(\bar{e}^r)^{n-k}r^*(\overline{\bar{e}}^r)^{n-\ell}\\
    &=\sum_{k,\ell=0}^nq^{(n)}_{k,\ell}\lambda^{k-\ell}r(e)^kr(\bar{e})^\ell r^*(\bar{e})^{n-k}r^*(e)^{n-\ell}  
    \end{align*}
    where $q_{k,\ell}^{(m)}$ are the same scalars as in the proof of Lemma \ref{lem:boundedness of sequence S_n and R_n}.
\end{cor}

\section{Main result}\label{sec:main result}
To establish the fullness of $\Gamma_q(\hr, U_t)$, we rely on the equivalent condition of the triviality of the relative commutant within the ultraproduct. We briefly recall the construction of the von Neumann ultraproduct. Let $\M$ be a von Neumann algebra equipped with a normal faithful state $\varphi$, and let $\omega$ be a free ultrafilter over the natural numbers. Denote by $\prod_{n\geq1}\M$ the algebra of all bounded sequences $(x_n)$ with each $x_n\in\M$. Let $I_\omega$ be the two sided ideal of $\prod_{n\geq1}\M$ given by
\begin{align*}
I_\omega=\left\{(x_n)\in\prod_{n\geq1}\M;\lim_{n\to\infty}\varphi(x_n^*x_n+x_nx_n^*)=0\right\}.
\end{align*}
The quotient $\M^\omega:=\prod_{n\geq1}\M/I_{\omega}$ is a well-defined von Neumann algebra, called ultraproduct of $\M$. A typical element in $\M^\omega$ will be denoted as $(x_n)_\omega$ which is the equivalence class of the sequence $(x_n)$ in the quotient.
 It is known  that $\M$ is a  full factor if and only if $\M' \cap \M^\omega = \mathbb{C}$ \cite{Con}. The proof of our main result on fullness uses this particular equivalent condition. In  what follows, one can also directly employ the equivalent condition of fullness as formulated in \cite[Corollary B]{Mar}.

\vspace{.1in}

\noindent
\textbf{\textit{Proof of Theorem 1.1.}}
If $\hr$ is finite-dimensional, then $\M_q$ is full by \cite[Theorem 2.7]{KSW}. If $\hr$ has a non-zero weakly mixing part or if the point spectrum $\sigma_p(A)$ of eigenvalues  of  the analytic generator $A$ of $\{U_t\}_{t\in\mathbb{R}}$ is infinite and bounded, then $\M_q$ is full by \cite[Theorem 6.2]{HI}. So
assume that $\hr$ is infinite-dimensional and there is a strictly decreasing  {sequence} of eigenvalues of $A$ converging to $0$.  

Now let $\omega$ be a fixed non-principal ultrafilter. Let $(X_k)_{\omega}\in \M_q'\cap\M_q^\omega$.
By replacing each $X_k$ by $X_k-\varphi(X_k)$, we may assume that $\varphi(X_k)=0$ for all $k\geq1$. In particular, we have $X_k\Omega, X_k^*\Omega\in \fqh\ominus\C\Omega$ for all $k\geq1$. We will show that $(X_k)_{\omega}=0$ in $\M_q^\omega$ i.e. $\lim_{k\to\omega}\|X_k\Omega\|_q=0=\lim_{k\to\omega}\|X_k^*\Omega\|_q$. 

Fix $\lambda\in \sigma_p(A)$ with $0<\lambda<1$, and  let $e\in \hc$ be such that $Ae=\lambda^{-1}e$ and $A\bar{e}=\lambda\bar{e}$ with $\|e\|_U=\lambda^{-\frac{1}{4}}$ and $\|\bar{e}\|_U=\lambda^{\frac{1}{4}}$. Consider the operator $S_n$ 
\begin{align*}
&S_n=\frac{1}{n}\sum_{m=1}^n(1-q)^{2m}W(\Bar{e}^me^m)c(\Bar{e})^mc(e)^m.
\end{align*}
Assume that $\lambda<(1+C_q^4)^{-2}$ so that the norm limit $S$ of $S_n$  exists and is invertible by Proposition \ref{prop:convergence of S_n=left wick formula}.
 We observe
\begin{align*}
    \lim_{k\to\omega}\|X_k\Omega\|_q^2&=\lim_{k\to\omega}\langle X_k\Omega, X_k\Omega\rangle_q=\lim_{k\to\omega}\lim_{n\to\infty}\langle X_k\Omega, S_nS^{-1}X_k\Omega\rangle_q=\lim_{n\to\infty}\lim_{k\to\omega}\langle X_k\Omega, S_nS^{-1}X_k\Omega\rangle_q
\end{align*}
where we could exchange the limit because $S_nS^{-1}$ converges in norm to identity as $n\to\infty$. Now we use the expressions $XX_k-X_kX\to0$ $*$-strongly as $k\to\omega$ for all $X\in \M_q$, and $X_kX'-X'X_k=0$ for all $X'\in \M_q'$ to get
\begin{align*}
    \lim_{k\to\omega}\|X_k\Omega\|_q^2&=\lim_{n\to\infty}\lim_{k\to\omega}\langle X_k\Omega, S_nS^{-1}X_k\Omega\rangle_q\\&=\lim_{n\to\infty}\lim_{k\to\omega}\frac{1}{n}\sum_{m=1}^n(1-q)^{2m}\langle X_k\Omega, W(\bar{e}^me^m)c(\bar{e})^mc(e)^mS^{-1}X_k\Omega\rangle_q\\
    &=\lim_{n\to\infty}\lim_{k\to\omega}\frac{1}{n}\sum_{m=1}^n(1-q)^{2m}\langle W(\bar{e}^me^m)X_k\Omega, c(\bar{e})^mc(e)^mS^{-1}X_k\Omega\rangle_q\\
    &=\lim_{n\to\infty}\lim_{k\to\omega}\frac{1}{n}\sum_{m=1}^n(1-q)^{2m}\langle X_kW(\bar{e}^me^m)\Omega, c(\bar{e})^mc(e)^mS^{-1}X_k\Omega\rangle_q\\
    &=\lim_{n\to\infty}\lim_{k\to\omega}\frac{1}{n}\sum_{m=1}^n(1-q)^{2m}\langle X_kW_r(\bar{e}^me^m)\Omega, c(\bar{e})^mc(e)^mS^{-1}X_n\Omega\rangle_q\\
    &=\lim_{n\to\infty}\lim_{k\to\omega}\frac{1}{n}\sum_{m=1}^n(1-q)^{2m}\langle X_k\Omega, W_r(\bar{e}^me^m)c(\bar{e})^mc(e)^mS^{-1}X_k\Omega\rangle_q\\
    &=\lim_{n\to\infty}\lim_{k\to\omega}\langle X_k\Omega,T_nS^{-1}X_k\Omega\rangle_q+\lim_{n\to\infty}\lim_{k\to\omega}\langle X_k\Omega,R_nS^{-1}X_k\Omega\rangle_q
\end{align*} 
where $T_n$ and $R_n$ are the following operators:
\begin{align*}
   & T_n=\frac{1}{n}\sum_{m=1}^n(1-q)^{2m}\left(W_r(\bar{e}^me^m)-r^*(\bar{e})^mr^*(e)^m-r(e)^mr(\Bar{e})^m\right)c(\bar{e})^mc(e)^m\\
    &R_n=\frac{1}{n}\sum_{m=1}^n(1-q)^{2m}(r^*(\bar{e})^mr^*(e)^m+r(e)^mr(\bar{e})^m)c(\bar{e})^mc(e)^m.
\end{align*}
Propositions \ref{prop:limit of cesaro sum of C(xi)^kc(eta)^kT_k} and \ref{prop: convergence of r^*(xi)r^*(eta)c(xi)c(eta)} tell us that the  operators $R_n$ (or $R_n^*$) when restricted to $\fqh\ominus\C\Omega$ converge to $0$ in norm as $n\to\infty$. Since $X_k\Omega\in \fqh\ominus\C\Omega$, it follows that
\begin{align*}
    \lim_{n\to\infty}\lim_{k\to\omega}\langle X_k\Omega,R_nS^{-1}X_k\Omega\rangle_q=0.
\end{align*}
Our next aim is to show that the operators $T_n$ are  bounded by $\frac{2C_q^6\sqrt{\lambda}}{1-\sqrt{\lambda}}$ as $n\to\infty$.
 First write using Corollary \ref{cor:formula for W_rbar{e}^ne^n} the following: 
\begin{align*}\label{eq:expression for right Wick formula}
    W_r(\Bar{e}^me^m)=\sum_{k,\ell=0}^mq^{(m)}_{k,\ell}\lambda^{k-\ell}r(e)^kr(\bar{e})^\ell r^*(\bar{e})^{m-k}r^*(e)^{m-\ell}. 
\end{align*}
 and then note that
 \begin{align*}
    W_r(\bar{e}^me^m)&-r^*(\bar{e})^mr^*(e)^m-r(e)^mr(\Bar{e})^m=\sum_{\ell=1}^m\sum_{k=0}^{m-1}q^{(m)}_{k,\ell}\lambda^{k-\ell}r(e)^kr(\bar{e})^\ell r^*(\bar{e})^{m-k}r^*(e)^{m-\ell}\\
    &+\sum_{\ell=0}^{m-1}q_{m,\ell}^{(m)}\lambda^{m-\ell}r(e)^m r(\bar{e})^\ell r^*(e)^{m-\ell}+\sum_{k=1}^{m}q_{k,0}^{(m)}\lambda^{k}r(e)^k  r^*(\bar{e})^{m-k}r^*(e)^{m}
    \end{align*}
so that the upper bound for $T_n$ can be given as follows:
\begin{align*}
\|T_n\|&\leq \frac{C_q^6}{n}\sum_{m=1}^n    \left(\sum_{\ell=1}^m\sum_{k=0}^{m-1}|q|^{(m-k)\ell}\lambda^{\frac{k-\ell}{2}}+\sum_{\ell=0}^{m-1}\lambda^{\frac{m-\ell}{2}}+\sum_{k=1}^{m}\lambda^{\frac{k}{2}}\right)
 \leq  {C_q^6}s_n+\frac{2C_q^6\sqrt{\lambda}}{1-\sqrt{\lambda}}.
\end{align*}
where $s_n$ is the sequence as in Lemma \ref{lem:convergence of the sclar sequence s_n}, which converges to $0$ by the same Lemma. Thus, we have shown that
\begin{align*}
\lim_{k\to\omega}\|X_k\Omega\|_q^2 &=\lim_{n\to\infty}\lim_{k\to\omega} \langle X_k\Omega, T_nS^{-1}X_k\Omega\rangle_q\\
&  \leq \left(\limsup_{n\to\infty}\|T_n\|\right)\;\|S^{-1}\|\;L\\
&\leq \frac{2LC_q^6\sqrt{\lambda}}{1-\sqrt{\lambda}} \cdot\frac{2C_q^2(1-\sqrt{\lambda})}{1-(1+C_q^4)\sqrt{\lambda}}
\end{align*}
where $L=\sup_{k\geq1}\|X_k\Omega\|^2$, and we have used the upper bound for $\|S^{-1}\|$ from Remark \ref{rem:estimate for the norm of S{-1}}.
By our assumption, we have a sequence of such $\lambda$'s which converges to $0$; hence it follows that $\lim_{k\to\omega}\|X_k\Omega\|_q=0$. A similar argument shows that $\lim_{k\to\omega}\|X_k^*\Omega\|_q=0$. This completes the proof.

\smallskip

\noindent {\bf Acknowledgments.}  We are grateful to Adam Skalski and Zhiyuan Yang  for going through the initial draft, and providing valuable suggestions; particularly, Zhiyuan suggested a shortened proof of Proposition \ref{prop: convergence of r^*(xi)r^*(eta)c(xi)c(eta)} by utilizing the operator $Q$. 
{We  thank the referee for their careful reading of the manuscript and for several useful comments.} 
This work was partially completed when M.K. was visiting S.W. at Harbin Institute of Technology in Aug-Sep 2023 for a month. M.K. extends appreciation to HIT for their warm hospitality and to S.W. for the invitation and financial support. 

M.K. was  partially supported by the National Science Center (NCN) grant no. 2020/39/I/ST1/01566. {The project is co-financed by the Polish National Agency for Academic Exchange within Polish Returns Program for Mateusz Wasilewski.} 
S.W. was partially supported by the NSF of China (No.12031004, No.12301161) and the Fundamental Research Funds for the Central Universities.  

\vspace{5 pt}
\includegraphics[scale=0.5]{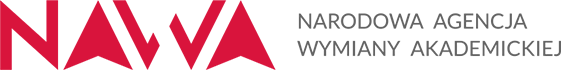}

\end{document}